[1994/12/01]
\documentclass[draft]{article}
\title{Linear Independence of Finite Gabor Systems Determined  by Behavior  at Infinity}
\author{John J. Benedetto and Abdelkrim Bourouihiya}
\date{\today}

\usepackage{amsmath,amsthm,amssymb,latexsym,amsfonts}
\usepackage{enumerate}

\chardef\bslash=`\\ 

\newtheorem{theorem}{Theorem}[section]
\newtheorem{lemma}[theorem]{Lemma}
\newtheorem{proposition}[theorem]{Proposition}
\newtheorem{corollary}[theorem]{Corollary}

\theoremstyle{definition}
\newtheorem{example}[theorem]{Example}

\theoremstyle{definition}
\newtheorem{definition}[theorem]{Definition}

\theoremstyle{definition}
\newtheorem{remark}[theorem]{Remark}

\numberwithin{equation}{section}

\newcommand{\eval}[2][\right]{\relax
  \ifx#1\right\relax \left.\fi#2#1\rvert}

\begin{document}
\maketitle
\markboth{Linear Independence of Finite Gabor Systems Determined  by Behavior  at Infinity}
{Linear Independence of Finite Gabor Systems Determined...}
\renewcommand{\sectionmark}[1]{}

 \begin{abstract}
 We prove that the HRT (Heil, Ramanathan,  and  Topiwala) conjecture holds for finite Gabor systems generated by square-integrable
 functions with certain behavior at infinity. These  functions include  functions ultimately decaying faster than any exponential function,
 as well as  square-integrable functions ultimately analytic and whose germs are in a Hardy field. Two classes of the latter type of functions are the set of square-integrable logarithmico-exponential functions and  the set of square-integrable Pfaffian functions. We also prove the HRT conjecture for certain finite Gabor systems generated by positive functions.
\end{abstract}
\section{ Introduction}
Let $L^2(\mathbb{R})$ be the space of square-integrable functions on the real line $\mathbb{R}$, and denote the $L^2$-norm of $f \in L^2(\mathbb{R})$
as $\left\Vert f \right\Vert_2$. If $g$ is a measurable function on $\mathbb{R}$ and  $\Lambda=\{(\alpha_k,\beta_k)\}_{k=1}^N$ is a set of finitely many distinct points in $\mathbb{R}^2$,   the \emph{finite Gabor system} generated by $g$ and $\Lambda$ is the set
$$\mathcal{G}(g, \Lambda)=\{e^{2\pi i\beta_kx}g(x-\alpha_k)\}_{k=1}^N.$$

In ~\cite{Hei2,Hei1}, the Heil, Ramanathan,  and  Topiwala (HRT) conjecture is stated as follows.
\begin{quote}
    Given $g\in L^2(\mathbb{R}) \setminus \{0\}$
and $\Lambda = \{(\alpha_k, \beta_k)\}_{k=1}^N$. Then $\mathcal{G}(g,\Lambda)$ is a linearly independent set of
functions in $L^2(\mathbb{R})$.
\end{quote}

We shall say that the HRT  conjecture holds for $g\in L^2(\mathbb{R}) \setminus \{0\}$ if the conjecture holds for $\mathcal{G}(g,\Lambda)$ for every set $\Lambda$ of finitely many distinct points in $\mathbb{R}^2$.

Despite the striking simplicity of the statement of the conjecture,
it remains open today. Some partial results, before our paper, include the following.
\begin{enumerate}
  \item If  $g\in L^2(\mathbb{R}) \setminus \{0\}$ is compactly supported, or supported
on a half-line, then the HRT conjecture holds for any value
$N$.
  \item  If $g(x) = p(x)e^{-x^2}$, where $p$ is a nonzero polynomial, then
the HRT conjecture holds for any value $N$.
\item The HRT conjecture  holds for any  $g\in L^2(\mathbb{R}) \setminus \{0\}$  if $N \leq 3$.
  \item If the HRT conjecture holds for a $g\in L^2(\mathbb{R}) \setminus \{0\}$ and $\Lambda$, then there exists an
$\varepsilon > 0$ such that the HRT conjecture holds for any $h \in L^2(\mathbb{R}) \setminus \{0\}$ satisfying $\| g-h
\|_2< \varepsilon$ using the same set $\Lambda$.
\item If the HRT conjecture holds
for $g \in L^2(\mathbb{R}) \setminus \{0\}$ and $\Lambda$, then there exists an
$\varepsilon >0$ such that the HRT conjecture holds for  $g$ and any set of N
points within $\varepsilon-$Euclidean distance of $\Lambda$.
\item  The HRT conjecture holds for any $g \in
L^2(\mathbb{R})\setminus \{0\}$ and any $\Lambda$ contained in some
translate of a full-rank lattice in $\mathbb{R}^2$. Such a lattice
has the form $A(\mathbb{Z}^2)$, where $A$ is an invertible matrix.
\end{enumerate}

Results (1)-(5) are published in the first paper ~\cite{Hei2} about the HRT  conjecture. Result (6) is due to Linnell ~\cite{Lin}. Other partial
results, where $\Lambda$ is not contained in a lattice, are published in ~\cite{Bal,Bow1,Bow2,Dem1,Dem2,Kut,Rze}.

We shall use the behavior of $g$ at infinity to prove that the HRT conjecture holds for several classes of functions. These include the following classes:
\begin{enumerate}
  \item The class of square-integrable functions whose \emph{germs} are analytic and are in a \emph{Hardy field} (Section 2), which includes the the class of  \emph{logarithmico-exponential} functions (see Example 2.3 and ~\cite{Bou,Har,Har2}) and
       the class of \emph{Pfaffian} functions (see Example 2.5 and ~\cite{Kho});
  \item The class of square-integrable functions $g$ such that
   $$\lim_{x \rightarrow \infty} \frac{g(x+\alpha)}{g(x)}$$
   exists for every positive real number $\alpha$ (Section 3);
  \item The class of functions $g$ decaying faster than any exponential function, i.e., $| g |$ is ultimately decreasing and $e^{tx}g(x) \in L^2(\mathbb{R})$, for every $t>0$ (Section 4).
\end{enumerate}

For the second class, we assume that the set of points $\{(\alpha_k, \beta_k)\}_{k=1}^N$, defining the finite Gabor
  system,  satisfies a \emph{difference condition for the second variable}, i.e., at least one of the $\beta_k$ is different from all the others. This class includes the set of differentiable and square-integrable functions $g$ such that
  $$\lim_{x \rightarrow \infty} \frac{g'(x)}{g(x)}$$
  exists in $\mathbb{C} \cup \{-\infty \}$.

Finally, we prove two theorems for finite Gabor systems generated by positive functions (Section 5). The first theorem states that the HRT conjecture holds for finite Gabor systems $\mathcal{G}(g, \{(\alpha_k,\beta_k)\}_{k=1}^N)$ if $g$ is ultimately  positive and $\{\beta_1,\ldots,\beta_N\}$ is linearly independent over $\mathbb{Q}$. The second theorem states that the HRT conjecture holds for every four element Gabor system  generated by an ultimately positive function $g$ if  both $g(x)$ and $g(-x)$ are ultimately decreasing.

In much of what follows we shall use the following propositions.
\begin{proposition}\label{trigo} Let $\beta_1,..., \beta_N \in \mathbb{R}$ be distinct, let $c_1, ..., c_N \in \mathbb{C}$, and let $E \subseteq \mathbb{R}$ have a positive Lebesgue measure. If
\begin{equation*}
\forall x \in E,\quad \sum_{k=1}^N c_k e^{2\pi i \beta_k x} =0,
\end{equation*}
then $c_1=c_2=.......=c_N=0$, see ~\cite{Hei2}.
\end{proposition}

The translation of $g \in L^2(\mathbb{R})$ by $\alpha \in \mathbb{R}$ is the function $T_\alpha g(x)=g(x-\alpha)$; the modulation of $g$ by $\beta \in \mathbb{R}$ is the function $M_\beta g(x)=e^{2\pi i \beta x}g(x)$; and the dilation of $g$ by $r \in \mathbb{R} \setminus \{0\}$ is the function $D_r g(x)= | r |^{\frac{1}{2}} g(rt)$.

\begin{proposition}\label{meta1} If $A$ is a linear transformation of $\mathbb{R}^2$ onto itself with $det A=1$, then there exits a unitary transformation $U_A : L^2(\mathbb{R}) \rightarrow L^2(\mathbb{R})$ such that
$$U_AM_bT_a = c_A(a,b)M_vT_u U_A,$$
where $(u,v) = A(a,b)$ and $c_A(a,b) \in \mathbb{C}$  has the property that $| c_A(a,b) | =1$.
\end{proposition}

The operators $U_A$ are \emph{metaplectic transforms}, and they form a group of linear transformations of $L^2(\mathbb{R})$ onto itself; we refer to ~\cite{Gro,Hei2,Hei1} for details. Translations, modulations, dilations, and the Fourier transform are examples of metaplectic transforms on $L^2(\mathbb{R})$.
\begin{proposition}\label{meta2}
Let $\mathcal{G}(g,\Lambda)$ be a finite Gabor system, and let $U: L^2(\mathbb{R}) \rightarrow L^2(\mathbb{R})$ be a metaplectic transform with associated linear transformation $A:\mathbb{R}^2\rightarrow \mathbb{R}^2$, i.e., $U=U_A$. Then, $\mathcal{G}(g,\Lambda)$ is a linearly independent set of functions in $L^2(\mathbb{R})$ if and only if
$\mathcal{G}(Ug,A(\Lambda))$ is a linearly independent set of functions in $L^2(\mathbb{R})$.
\end{proposition}

Notationally, $\mathcal{S}(\mathbb{R})$ is the Schwartz space of rapidly decreasing infinitely differentiable functions on $\mathbb{R}$, $\widehat{g}$ denotes the Fourier transform of $g$, and $|E|$ is the Lebesgue measure of $E \subseteq \mathbb{R}$.

\section{Hardy Fields and the HRT Conjecture}
Given a property $P$ defined on a set $X \subseteq \mathbb{R}$, which includes an interval $(a, \infty)$. We say that $P(x)$ \emph{ultimately} holds  if there  is $x_0 \in (a, \infty)$ such that
$P(x)$ holds for all $x > x_0$.

Let $F$ be the set of all functions $f:X_f \rightarrow \mathbb{R}$ such that $(a_f, \infty) \subseteq X_f \subseteq \mathbb{R}$ for some
$a_f \in \mathbb{R}$. We define an equivalence relation $\sim$  on $F$ by writing $f \sim g$ to mean $f(x)= g(x)$ for all $x$ greater than some
$a > \max(a_f,a_g)$, i.e., $f$ is ultimately equal to $g$. The equivalence class associated with $f \in F$ is denoted by $germ(f)$. Addition and multiplication of functions are compatible with
 respect to $\sim$, and so the set $\mathcal{F}= \{ germ(f): f \in F\} $ is a commutative ring.

\begin{definition}A subring $\mathcal{H}$ of $\mathcal{F}$ is a \emph{Hardy field} if it is a field and it is closed under differentiation.\end{definition}

Some known properties of Hardy fields are collected in the  following proposition.

\begin{proposition}\label{prop:hardyfield} Let $E$ be a set of real-valued functions on $\mathbb{R}$ such that the germs of all functions in $E$ are in a Hardy field $\mathcal{H}$.
\begin{enumerate}[(a)]
  \item Every function in $E$ is ultimately  strictly monotone or constant and ultimately has a constant sign.
  \item If $f$ and $g$ are in $E$ and have nonzero germs, then the limit at infinity of $f/g$ or $g/f$ is finite.
If the limit at infinity of $f/g$ is finite we say that $f$ is asymptotically smaller than $g$ and we write $germ(f) \preceq germ(g)$.
  \item The Hardy field $\mathcal{H}$ is well-ordered with respect to the relation $\preceq$.
\end{enumerate}
\end{proposition}

It is elementary to see that the germs of rational functions on $\mathbb{R}$ form a Hardy  field.

\begin{example}
The space, $LE$, of \emph{logarithmico-exponential} functions is the smallest set of ultimately defined real valued functions containing the identity function $I(x)=x$ and every constant function $C(x)=c \in \mathbb{R}$ and closed under the following operations: $f,g \in LE$ implies $f \pm g, fg, f/g \in LE$; if $f \in LE$ then $e^f \in LE$; if $f \in LE$ is ultimately positive then $\log f \in LE$; and if $f \in LE$ then $\sqrt[n]{f} \in LE$, for every integer $n>0$. For example, $\exp (\sqrt{\log x}/\log \log x) \in LE.$

Hardy introduced the class $LE$ in 1910 ~\cite{Har,Har2}; and he proved the fundamental fact that \emph{the germs of $LE$ functions form a Hardy field}. His motivation was to interpret the idea of a scale of infinities.

The apparent specificity of the space, $LE$, is in contrast to its broad applicability. For example, $LE$ and more general Hardy fields play a role in model theory (logic), e.g., ~\cite{Kuh}, time complexity in theoretical computer science, e.g., ~\cite{Cai}, differential equations, e.g., ~\cite{Hart,Mar}, and, of course, Tauberian Theory, e.g., ~\cite{Kar,Kor}.
\end{example}

\begin{theorem}\label{thm:hrthardyfield}  Let $E$ be a real vector space of real-valued functions on $\mathbb{R}$ such that each $f \in E$ has the properties that it is ultimately analytic and
$germ(f)$ is in a Hardy field $\mathcal{H} = \mathcal{H}_E$. Assume that $E$ is closed under all real translations. Let $\mathcal{E}$ be the complex vector space
generated by $E$. The HRT conjecture holds for $\mathcal{G}(g, \Lambda)$ for each $g \in \mathcal{E} \cap L^2(\mathbb{R})\setminus \{0\}$ and arbitrary $\Lambda$.
\end{theorem}

\begin{proof}
{\it i. }Let $g \in \mathcal{E} \cap L^2(\mathbb{R}) \setminus \{0\}$ and suppose that the HRT conjecture does not hold for $\mathcal{G}(g,
\Lambda)$ for some finite subset $\Lambda=\{(\alpha_k,\beta_k)\}_{k=1}^N$. In part \emph{ii}, we prove that we may assume without loss of generality that  $g$ is analytic on $\mathbb{R}$. After a convenient relabeling in part \emph{iii}, we use the fact that a Hardy field is well-ordered with respect to the relation $\preceq$ (Proposition \ref{prop:hardyfield}) in part \emph{iv}, and this will yield the desired contradiction.

{\it ii. }Assume that the HRT conjecture fails for some $\Lambda$. Using Proposition 1.3 and relabeling,  we suppose without loss of generality that
\begin{eqnarray}
\sum_{k=1}^M c_k e^{2 \pi i \beta_k x} g(x)= \sum_{k=M+1}^N c_k e^{2 \pi i \beta_k x}g(x+ \alpha_k) \quad \mbox{a.e},
\end{eqnarray}
where $\alpha_1, \ldots, \alpha_N>0$, $c_1,\ldots, c_N \in \mathbb{C} \setminus \{0\}$, $\beta_{1}, \ldots, \beta_N \in \mathbb{R}$, and $\beta_{1}, \ldots, \beta_M $ are distinct. Then, we compute
\begin{eqnarray*}
  &&\prod_{j=M+1}^N p(x+\alpha_j) p(x)g(x)= \sum_{k=M+1}^N c_k e^{2 \pi i \beta_k x}\prod_{j=M+1}^N p(x+\alpha_j)g(x+ \alpha_k)\\
  &&= \sum_{k,l=M+1}^N c_k  c_le^{2 \pi i \beta_l \alpha_k}e^{2 \pi i (\beta_k+\beta_l) x}p_k(x)
  g(x+ \alpha_k + \alpha_l) \quad \mbox{a.e},\\
\end{eqnarray*}
where
\begin{eqnarray*}
p(x) = \sum_{m=1}^M c_m e^{2 \pi i \beta_m x} \ \  \mbox{ and } p_k(x)=\prod_{j \in \{M+1,\ldots,N\} \setminus \{k\}}p(x+\alpha_j),
\end{eqnarray*}
for each $k \in \{M+1,\ldots,N\}$.

We already know that $g$ is ultimately analytic, i.e.,  $g$ is analytic on an interval
 $(A, \infty)$, for some real number $A$. Let $a \in \mathbb{R}$. Iterating the above procedure, as many times as needed, we can find an equality similar to (2.1)  with $a+\alpha_k> A$, for each $k \in \{M+1,\ldots,N\}$. Therefore, the right-hand side of (2.1) is analytic for all $x>a$. In other words, we proved the following. For each $a \in \mathbb{R}$ there exist $P_a$ and $G_a$ such that $g(x) = G_a(x)/P_a(x)$ for almost all $x>a$, where $P_a$ is a trigonometric polynomial and $G_a$ is a linear combination of time-frequency shifts of $g$ that are analytic on $(a,\infty)$. Therefore, $P_a$ and $G_a$ are analytic on $(a,\infty)$, and  hence, for each $x_0 >a$, there is an open interval $I$ containing $x_0$ and there is $n \in \mathbb{Z}$ such that
\begin{eqnarray}
  \forall x \in I, \quad \frac{G_a(x)}{P_a(x)} = (x-x_0)^n H_a(x),  \nonumber
\end{eqnarray}
where $H_a$ is analytic and never vanishes on $I$. Since $g$ is square-integrable and $g(x) = G_a(x)/P_a(x)$ for almost all $x \in I$, then,  $G_a/P_a \in L^2(I)$, and so $n \geq 0$. Therefore, $G_a/P_a$ is analytic on $I$, and, consequently, $G_a/P_a$ is analytic on $(a,\infty)$.

If $a,b \in \mathbb{R}$, then ,  $G_a(x)/P_a(x) = G_b(x)/P_b(x)$ for almost all $x > \max(a,b)$; and the fact that $G_a/P_a$ and $ G_b/P_b$ are analytic on $(\max(a,b),\infty)$ implies that $G_a(x)/P_a(x) = G_b(x)/P_b(x)$ for all $x > \max(a,b)$. Thus,    $\widetilde{g}(x)=G_a(x)/P_a(x)$, where $a$ is any real number less than $x$, is a well defined function that is analytic on $ \mathbb{R}$; and for all $n \in \mathbb{Z}$, we have $\widetilde{g}(x) = g(x)$ for almost all $x>n$, i.e., $\mid \{x: \widetilde{g}(x) \neq g(x) \mbox{ and } x >n   \} \mid =0$  for each  $n \in \mathbb{Z}$, and so
\begin{eqnarray*}
  \mid \{x \in \mathbb{R}: \widetilde{g}(x) \neq g(x) \}\mid = \mid \bigcup_{n \in \mathbb{Z}}\{x: \widetilde{g}(x) \neq g(x) \mbox{ and } x >n   \} \mid =0,
\end{eqnarray*}
i.e., $\widetilde{g}= g$ almost everywhere. This with the fact that $\widetilde{g}$ is analytic on $ \mathbb{R}$ imply that (2.1) holds for $\widetilde{g}$ everywhere. Therefore, without loss of generality, we assume for the rest of the proof  that $g$ is analytic on $\mathbb{R}$ and that (2.1) holds everywhere.

{\it iii. }After relabeling, we may suppose that
\begin{eqnarray}\label{eqn:2.2}
  \sum_{k=1}^N e^{2\pi i \beta_kx} g_k(x) =0,
\end{eqnarray}
where $\beta_1,\ldots, \beta_N \in \mathbb{R}$ are distinct and, for each $k=1, 2,\ldots, N$,
\begin{eqnarray*}
  g_k(x) = \sum_{n=1}^{N_k} c_{(k,n)}g(x-\alpha_{(k,n)}),
\end{eqnarray*}
where $c_{(k,1)}, c_{(k,2)},\ldots,c_{(k,N_k)} \in \mathbb{C} \setminus \{ 0 \}$ and $\alpha_{(k,1)},\alpha_{(k,2)},\ldots,\alpha_{(k,N_k)} \in \mathbb{R}$.

{\it iv. }By Proposition \ref{trigo} and taking the Fourier transform we note that, for each $k=1, 2,\ldots, N$, $\{T_{\alpha_{(k,n)}}g\}_{n=1}^{N_k}$ is a linearly independent set of functions, cf. ~\cite{Hei1,Ros}. Thus, $g_k$ is not identically
equal to zero. Using the fact that $g_k$ is analytic, we obtain that $g_k$ is not ultimately equal to zero. Therefore, and since $E$ is closed under
translations, there are $f_k,h_k \in E$ such that $g_k=f_k+ih_k$ for which  $germ(f_k) \neq 0$ or $germ(g_k) \neq 0$. In particular  if $germ(f)$ is
the maximum of $\{ germ(f_k), germ(h_k): k=1,2,\ldots, N\}$ with respect to the relation $\preceq$, then $germ(f) \neq 0$.

Equation (2.2) can be rewritten as
\begin{eqnarray*}
  \sum_{k=1}^N e^{2\pi i \beta_kx} (f_k(x)+i g_k(x)) =0,
\end{eqnarray*}

Now let $\{x_n\} \subseteq \mathbb{R}$ be a  sequence  converging to infinity such that
\begin{eqnarray*}
  \forall k = 1, \ldots, N, \quad \lim_{n \rightarrow \infty} e^{2\pi i \beta_kx_n} = L_k.
\end{eqnarray*}
Then, we compute
\begin{eqnarray*}
  \lim_{n \rightarrow \infty} \sum_{k=1}^N e^{2\pi i \beta_k(x+x_n)} \frac{f_k(x+x_n)+i g_k(x+x_n)}{f(x+x_n)} =0.
\end{eqnarray*}
Using Proposition \ref{prop:hardyfield}, we obtain
\begin{eqnarray*}
  \lim_{n \rightarrow \infty}  \frac{f_k(x+x_n)+i g_k(x+x_n)}{f(x+x_n)} =z_k,
\end{eqnarray*}
where $z_1, z_2, . . . , z_{N} \in \mathbb{C }$. Therefore, we have
\begin{eqnarray*}
  \sum_{k=1}^N z_k L_k e^{2\pi i \beta_kx}=0.
\end{eqnarray*}
This contradicts Proposition \ref{trigo}, because $\beta_1, \ldots, \beta_N \in \mathbb{R}$ are distinct, $L_k \neq 0$ for each $k \in \{1, 2, \ldots,N \}$,
and $z_{k} \neq 0$, for at least one $k \in \{ 1, 2, \ldots, N \}$.
\end{proof}

\begin{example}
\emph{(a)} The class of $LE$-functions satisfies the conditions of Theorem \ref{thm:hrthardyfield}. Thus, the HRT conjecture holds for every
$g \in \mathcal{E} \cap L^2(\mathbb{R}) \setminus \{0\}$, where $\mathcal{E}$ is the complex vector space generated by $LE$-functions. For example, the HRT conjecture holds for the function
$$g(x)=\frac{e^{-|x|}}{1+\sqrt{\mid x \mid}}+  \frac{\log|x|}{1+ix\log|x|}.$$

\emph{(b)} Let $E$ be the class of real-valued analytic functions $g_1$ defined as follows: $g_1 \in E$ if there exist $N-1$ analytic functions $g_2,\ldots, g_N$ such
that $(g_1,g_2,\ldots,g_N)$ is a solution of a system of first degree differential equations having the form
$$\frac{dy_n}{dt}=\sum_{k=1}^N p_k(t,y_1,\ldots,y_N), \ \ n=1,\ldots,N,$$
where $p_1, \ldots,p_N$ are polynomials of $(N+1)$ variables. The elements of $E$ are called
\emph{Pfaffian functions}; and the germs of  such functions form a Hardy field ~\cite{Kho}.
Thus, by Theorem \ref{thm:hrthardyfield}, the HRT conjecture holds for any square-integrable linear combination (with complex coefficients) of functions in $E$.
\end{example}

\begin{remark}
Pfaffian functions were introduced by Khovanskii ~\cite{Kho}. They include many, but not all, elementary functions, as well as some special functions.
Khovanskii also proved that the germs of functions built from LE and trigonometric functions form a Hardy field, provided that the arguments of the
$sine$ and $cosine$ functions are bounded ~\cite{Kho}, e.g.,
$$ f(x)=\sin(\frac{x}{1+x^2})e^{-\sqrt{|x|}}.$$
Thus, by Theorem \ref{thm:hrthardyfield}, the HRT conjecture holds for such functions.

There are other classes of functions satisfying the conditions of Theorem \ref{thm:hrthardyfield}. These include \emph{D-finite functions} defined in ~\cite{Sta}.

Liouville proved ``elementary integrability" criteria allowing one to assert that certain integrals, most famously $\int e^{-x^2}dx$, cannot be expressed ``in elementary terms"; of course, ``elementary" has to be defined in a precise way, see ~\cite{Con,Ros1,Ros2}. We mention this since, if we replace $E$
in Theorem \ref{thm:hrthardyfield} by a space generated by $E$ and  the primitives of all functions in $E$, we can still conclude that the linear independence conclusion holds for
finite linear combinations of square integrable functions belonging to the new space.
\end{remark}

The proofs of the following theorems are similar to the proof of Theorem \ref{thm:hrthardyfield}.

\begin{theorem}\label{thm:hrthardyfield2} Let $f \in  L^2(\mathbb{R}) \setminus \{0\}$ have the properties that $f$ is  analytic on $\mathbb{R}$ and $germ(f)$ is in a Hardy field
$\mathcal{H}$ that is closed under all real translations. Assume $h \in  L^2(\mathbb{R})$ satisfies the condition,
\begin{eqnarray*}
\lim_{x \rightarrow \infty } \frac{h(x)}{f(x)} = 0.
\end{eqnarray*}
The HRT conjecture holds for $\mathcal{G}(f+h, \Lambda)$, where $\Lambda$ is arbitrary.
\end{theorem}

\begin{proof}
If the HRT conjecture does not hold for $\mathcal{G}(f+h,\Lambda)$, for some finite set $\Lambda \subset \mathbb{R}^2$,
we may suppose that
\begin{eqnarray*}
  \sum_{k=1}^N e^{2\pi i \beta_kx} (f_k(x)+h_k(x)) =0 \quad \mbox{a.e},
\end{eqnarray*}
where $\beta_1,\ldots, \beta_N$ are distinct real numbers and, for each $k=1, 2,\ldots, N$,
\begin{eqnarray*}
  f_k(x) = \sum_{n=1}^{N_k} c_{(k,n)}f(x-\alpha_{(k,n)}) \ \ \  \mbox{ and } \ \ \ h_k(x) = \sum_{n=1}^{N_k} c_{(k,n)}h(x-\alpha_{(k,n)}),
\end{eqnarray*}
where $ c_{(k,1)}, c_{(k,2)},\ldots,c_{(k,N_k)} \in \mathbb{C} \setminus \{ 0 \}$ and $\alpha_{(k,1)},\alpha_{(k,2)},\ldots,\alpha_{(k,N_k)} \in \mathbb{R}$.

Using an argument similar to the steps in the proof of Theorem \ref{thm:hrthardyfield}, we can prove that $ f_k$ is not ultimately equal to zero, for each $k \in \{1,2,\ldots, N\}$.
Then $ f_k = u_k+iv_k$, where $germ(u_k),germ(v_k) \in \mathcal{H}$ and $germ(u_k)\neq 0$ or $germ(v_k)\neq 0$, for each $k \in \{1,2,\ldots, N\}$.
In particular, if $germ(u)$ is the maximum of $\{ germ(u_k), germ(v_k): k=1,2,\ldots, N\}$ with respect to the relation $\preceq$, then $germ(u) \neq 0$. Therefore, we obtain a contradiction as in the last steps in the proof of Theorem \ref{thm:hrthardyfield}.
\end{proof}

\begin{corollary}
Let $f \in L^2(\mathbb{R})\setminus \{0\}$ be a rational function, let $h \in \mathcal{S}(\mathbb{R})$,  and take $(a,b) \in \mathbb{R}^2\setminus \{(0,0)\}$ and $t>0$. The HRT conjecture holds for $\mathcal{G}(h(x)+ae^{-t|x|} + bf(x), \Lambda)$, where $\Lambda$ is arbitrary.
\end{corollary}
\begin{proof} The case where  $g=h+ ae^{-t\mid x \mid}$ can be obtained by taking the Fourier transform of $g$. The other cases are immediate
consequences of Theorem \ref{thm:hrthardyfield2}.
\end{proof}

\begin{theorem} Let $g \in L^2(\mathbb{R}) \setminus \{0\}$ have the property that $g$ is analytic on $\mathbb{R} \setminus E$, where $E \neq \emptyset$ and $\text{card}(E) < \infty$. The HRT conjecture holds for $\mathcal{G}(g, \Lambda)$, where $\Lambda$ is arbitrary.
\end{theorem}
\begin{corollary}
Let $\varepsilon >0$ and let  $h \in L^2(\mathbb{R})$ be analytic on $\mathbb{R}$. The HRT conjecture holds for $\mathcal{G}(e^{-|x|^\varepsilon}+h(x), \Lambda)$, where $\Lambda$ is arbitrary.
\end{corollary}
\section{The HRT Conjecture for the Ratio-Limit Case}
\begin{definition}A measurable function $g$ on $\mathbb{R}$ has the \emph{ratio-limit}
 $l_g(\alpha) \in \mathbb{C} \cup \{\pm \infty\}$  at  $\alpha \in \mathbb{R}$ if
\begin{eqnarray*}
\lim_{x \rightarrow \infty} \frac{g(x+\alpha)}{g(x)} = l_g(\alpha).
\end{eqnarray*}
\end{definition}

Some elementary properties of ratio-limits are collected in the following proposition.

\begin{proposition}\label{prop:ratiolimit} Let $g$ be a measurable function on $\mathbb{R}$ having the finite ratio-limit $l_g(\alpha)$ at $\alpha \in \mathbb{R}$.
\begin{enumerate}[(a)]
\item The functions $T_a g$, $M_\beta g$, and $D_rg$ have a ratio-limit at $\alpha$, and, in fact,
$$l_{T_a g}(\alpha)=l_g(\alpha), \ \ l_{M_\beta g}(\alpha)=e^{2\pi i \beta  \alpha}l_g(\alpha), \mbox{ and } \ \ l_{D_r g}(\alpha)=l_g(r\alpha).$$

\item Let $h$ be a measurable function on $\mathbb{R}$ and assume that $h \sim g$. Then, $h$ has the ratio-limit $l_h(\alpha)$ at $\alpha$, and $l_h(\alpha)=l_g(\alpha)$.

\item Let $f$  be a measurable function on $\mathbb{R}$ and assume that $f$ has the finite ratio-limit $l_f(\alpha)$ at $\alpha$. Then, the function $fg$ has the ratio-limit $l_{fg}(\alpha)$ at $\alpha$, and  $l_{fg}(\alpha)=l_f(\alpha)l_g(\alpha)$.

\item Assume that $g$ has the finite ratio-limit $l_g(\beta)$ at $\beta \in \mathbb{R}$. Then, $g$ has the ratio-limit $l_g(\alpha+\beta)$ at $\alpha + \beta$, and $l_g(\alpha+\beta)=l_g(\alpha)l_g(\beta)$.
\end{enumerate}
\end{proposition}
\begin{proof} Each of the proofs is elementary. To illustrate we shall prove part \emph{(d)}. Assume that $g$ has the finite ratio-limit $l_g(\alpha)$ at $\alpha \in \mathbb{R}$ and the finite ratio-limit $l_g(\beta)$ at $\beta \in \mathbb{R}$. Therefore, we have
\begin{eqnarray*}
  \lim_{x \rightarrow \infty} \frac{g(x+\alpha+\beta)}{g(x)}&=& \lim_{x \rightarrow \infty} \frac{g(x+\alpha+\beta)}{g(x+\beta)}\frac{g(x+\beta)}{g(x)}
  = l_g(\alpha)l_g(\beta).
\end{eqnarray*}
Thus, $g$ has the ratio-limit $l_g(\alpha+\beta)$ at $\alpha + \beta$, and $l_g(\alpha+\beta)=l_g(\alpha)l_g(\beta)$.
\end{proof}

\begin{proposition}
Let $g \in L^2(\mathbb{R})$. Suppose that $g$ has the ratio-limit $l_g(\alpha)$ at each $\alpha >0$.
Then, there exists $0\leq a \leq 1$ such that
\begin{eqnarray*}
\forall \alpha >0, \quad |l_g(\alpha)|= a^\alpha.
\end{eqnarray*}
\end{proposition}

\begin{proof}
Let $l(\alpha)= |l_g(\alpha)|$ and $l(1)=a$. Suppose that
$l(\alpha)>1$ for some $\alpha>0$. Then, there exists
$A>0$ such that
\begin{eqnarray}
\forall x>A, & \mid g(x+\alpha)\mid > \mid g(x) \mid. \nonumber
\end{eqnarray}
Consequently, we have
\begin{eqnarray}
\int _A ^\infty |g(x+\alpha)|^2 dx > \int _A ^\infty |g(x)|^2 dx=
\int _A ^{A+\alpha} |g(x)|^2 dx+ \int _A ^\infty |g(x+\alpha)|^2 dx,
\nonumber
\end{eqnarray}
yielding the contradiction,
\begin{eqnarray}
0 > \int _A ^{A+\alpha} |g(x)|^2 dx. \nonumber
\end{eqnarray}
Therefore, $0\leq l(\alpha)\leq 1$ for all
$\alpha \geq 0$, and so, in particular, $0\leq a \leq 1$.

Using Proposition 3.2, we can  prove that $l(r)=a^r$ for all rational numbers  $r>0$. Further, note that if $\alpha>\beta\geq 0$, then
$l(\alpha)=l(\beta)l(\alpha - \beta)\leq l(\beta)$. Thus, the function $l$ is decreasing on $(0,\infty)$.

If $\alpha>0 $, then there exist two sequences, $\{s_n\}$ and $\{r_n\}$, of
positive rational numbers converging to $\alpha$ and satisfying the inequalities, $s_n\leq \alpha \leq r_n,$
for each $n$. Thus, since $l$ is decreasing, we have
\begin{eqnarray*}
\forall n \geq 1, & a^{r_n}\leq l(\alpha) \leq a^{s_n}.
\end{eqnarray*}
Letting $n$ tend to infinity, we obtain $l(\alpha)= a^\alpha$, and the proof is complete by once again invoking Proposition \ref{prop:ratiolimit}.
\end{proof}
\begin{remark}
\emph{Regularly varying functions} are real-valued functions $\varphi$, defined on $(0,\infty)$, having the property that
$\lim_{x \rightarrow \infty} \varphi(\lambda x)/\varphi(x)$ exists for each $\lambda >0$. They were introduced and used by
J. Karamata  to prove his Tauberian theorem ~\cite{Kar}, cf. the notion of \emph{slowly oscillating functions} which also play a basic role in Tauberian theory, ~\cite{Ben}, Sections 2.3.4 and 2.3.5. If a real-valued function $g$ has the ratio-limit $l_g(\alpha)$ at each $\alpha \in \mathbb{R}$, it is said to be \emph{additively regularly varying}, i.e.,  the function $\varphi(x)=g(\log x)$ is regularly varying.
\end{remark}

\begin{lemma}Let $g$ be a complex valued function on $\mathbb{R}$ for which the logarithmic derivative exists on $[a,b]$. Then, we have
$$ \frac{g(b)}{g(a)}= \exp \left (\int_a^b \frac{g'(x)}{g(x)} dx \right).$$
\end{lemma}
\begin{proof} Since the logarithmic derivative $g$ exists on $[a,b]$, the function $g$ is continuous and $g(x) \neq 0$ for all $x \in [a,b]$. Therefore, $g([a,b])$ is a compact subset of $\mathbb{C}\setminus \{0\}$, and so we can choose $\theta \in \mathbb{R}$ for which the open set $U= \mathbb{C} \setminus  \{ te^{i\theta}: t\geq 0\} $ contains $g([a,b])$. If we denote by $L_U( z)$ the branch of the complex logarithm defined on $U$, then we compute
$$ L_U\left( \frac{g(b)}{g(a)}\right)= \int_a^b \frac{g'(x)}{g(x)} dx,$$
and so
$$ \frac{g(b)}{g(a)}= \exp \left (\int_a^b \frac{g'(x)}{g(x)} dx \right).$$

\end{proof}

\begin{proposition}Let $g$ be a complex valued function for which the logarithmic derivative ultimately exists.
\begin{enumerate}[(a)]

\item If the logarithmic derivative of $g$ has a finite limit $l$ at infinity,
then $g$ has the ratio-limit $l_g(\alpha)= e^{l\alpha}$ at each $\alpha >0$.
\item If the limit of the logarithmic derivative of $g$  is $-\infty$,
then $l_g(\alpha)=0$, for all $\alpha>0$.
\end{enumerate}
\end{proposition}

\begin{proof} \emph{(a)} Let $\alpha >0$. Assume that the logarithmic derivative of $g$ has a limit $l \in \mathbb{C}$ at infinity. Therefore, if  $ \epsilon >0$, then there exists $A>0$ for which
$$\forall x>A, \quad \mid \frac{g'(x)}{g(x)} - l \mid < \frac{\epsilon}{\alpha},$$
and so
$$\int_{x}^{x+\alpha}\mid \frac{g'(t)}{g(t)} - l \mid dt < \epsilon. $$
Therefore, we have
$$\forall x>A, \quad  \mid \int_{x}^{x+\alpha} \frac{g'(t)}{g(t)}dt - l\alpha \mid  < \epsilon.$$
Consequently, we compute
$$\lim_{x\rightarrow \infty } \int_{x}^{x+\alpha} \frac{g'(t)}{g(t)}dt = l\alpha,$$
and hence, using Lemma 3.5, we obtain
$$\lim_{x\rightarrow \infty } \frac{g(x+\alpha)}{g(x)} = e^{l\alpha}.$$

Using a similar argument, we can prove part {\it (b)}.
\end{proof}

\begin{example}
\emph{a.} Rational functions $f$ have the ratio-limits $l_f(\alpha)=1$ at each  $ \alpha \in \mathbb{R}$.

\emph{b.} Measurable functions $f$ on $\mathbb{R}$ that are analytic at $\infty$ have the ratio-limits $l_f(\alpha)=1$ at each  $ \alpha \in \mathbb{R}$.

\emph{c.} For all $\epsilon >0$, the function $g(x)= e^{-|x|^\epsilon}$ has the
ratio-limit $l_g(\alpha)$ for all $\alpha>0$. In this case, we can compute that
\[l_g(\alpha) = \left\lbrace
  \begin{array}{c l}
  1, &  \mbox{ \emph{if} } \hspace{.2in} 0<\epsilon <1,  \\
   e^{-\alpha}, &  \mbox{ \emph{if} }  \hspace{.2in} \epsilon =1,  \\
   0, &  \mbox{ \emph{if} }   \hspace{.2in} \epsilon >1.
  \end{array}
\right. \]

\emph{d.} Trigonometric functions do not have ratio-limits at each $ \alpha \in \mathbb{R}$, e.g., the function $h(x)=\sin (2 \pi x)$ does not have a ratio limit at  $\sqrt{2}$.
\end{example}

Let $\Lambda = \{(\alpha_k, \beta_k)\}_{k=1}^N \subseteq \mathbb{R}^2$ be a set of distinct points. We say that $\Lambda $
satisfies \emph{the difference condition for the second variable} if there
exists $k_0 \in \{1,\ldots,N\}$ such that $\beta_{k} \neq \beta_{k_0}$, whenever $k \neq k_0$. The difference condition for the first
variable is similarly defined.

\begin{lemma}
Let $P$ be a property that holds for almost every $x \in \mathbb{R}$. For every sequence $\{u_n\}_{n \in \mathbb{N}} \subset \mathbb{R}$, there exists $E \subseteq \mathbb{R}$ such that $\mid \mathbb{R} \setminus E \mid =0$ and $P$  holds for $x+u_n $ for each $(n,x) \in \mathbb{N} \times E$.
\end{lemma}

\begin{proof} If $E=\bigcap_{n \in \mathbb{N}}\{ x : P(x+u_n) \mbox{ holds } \}$, then $P$  holds for $x+u_n $ for each $(n,x) \in \mathbb{N} \times E$. We know that $\mid \{ x : P(x+u_n) \mbox{ fails } \} \mid = 0$, for each $n \in \mathbb{N}$, and so $\mid \bigcup_{n \in \mathbb{N}} \{ x : P(x+u_n) \mbox{ fails } \} \mid = 0$, i.e., $\mid \mathbb{R} \setminus E \mid = 0$.
\end{proof}

\begin{theorem}
Let $g\in L^2(\mathbb{R})$ have the ratio-limit $l_g(\alpha)$ at every $\alpha >0$, and let
$\Lambda=\{(\alpha_k,\beta_k) \}_{k=1}^N \subseteq \mathbb{R}^2 $. The HRT conjecture holds for
$\mathcal{G}(g,\Lambda)$ in the following cases:

\begin{enumerate}[(a)]
\item $l_g(1)=0$ and $\Lambda$ is any finite subset of $\mathbb{R}^2$; and

\item $l_g(1)\neq 0$ and $\Lambda$ satisfies the difference condition for the second variable.
\end{enumerate}
\end{theorem}

\begin{proof}
Note that since $g$ has a ratio-limit, then $g$ is ultimately nonzero. Suppose that the HRT conjecture fails. We shall obtain a contradiction for each of the two cases.

{\it (a)} If $l_g(1)=0$, then, by Proposition 3.3, $l_g(\alpha) =0$ for all $\alpha>0$. Using Proposition 1.3, without loss of generality we suppose that
\begin{eqnarray}
\sum_{k=1}^M c_k e^{2 \pi i \beta_k x}g(x)= \sum_{k=M+1}^N c_k e^{2
\pi i \beta_k x}g(x+ \alpha_k) \quad \text{a.e.}, \nonumber
\end{eqnarray}
where $c_1,\ldots, c_M \in \mathbb{C}\setminus \{0\}$, $c_{M+1},\ldots, c_N \in \mathbb{C}$, $\alpha_k >0$ for all $k=M+1, \ldots, N$,
$\beta_1, \ldots, \beta_N \in \mathbb{R}$, and $\beta_{1}, \ldots, \beta_M \in \mathbb{R}$ are distinct.

Let $\{x_n\}_{n \in \mathbb{N}}$ be a positive sequence converging to infinity, with the property that the sequence $\{e^{2 \pi i \beta_k x_n}\}_{n \in \mathbb{N}}$ converges to a limit $L_k$
for each $k \in \{ 1, \ldots, N\}$. Then, $|L_k|=1$, and, in particular, $L_k \neq 0$ for each $k \in \{ 1, \ldots, N\}$. By Lemma 3.8, there is $E \subseteq \mathbb{R}$ such that $\mid \mathbb{R} \setminus E \mid =0$ and, for all $(n,x) \in \mathbb{N} \times E$,
\begin{eqnarray}
\sum_{k=1}^M c_k e^{2 \pi i \beta_k (x + x_n)}g(x + x_n)= \sum_{k=M+1}^N c_k e^{2\pi i \beta_k (x + x_n)}g(x + x_n+\alpha_k). \nonumber
\end{eqnarray}

Let $x \in E$ be fixed. Since $g$ is ultimately nonzero, then, there is $n_0>0$ such that  $g(x + x_n) \neq 0$ for each $n>n_0$,
and so we can write
\begin{eqnarray*}
\sum_{k=1}^M c_k e^{2 \pi i \beta_k (x+x_n)}= \sum_{k=M+1}^N c_k e^{2 \pi i \beta_k (x+x_n)}\frac{g(x+x_n+\alpha_k)}{g(x+x_n)}.
\end{eqnarray*}
Hence, letting $n$ tend to infinity in the last equality, we obtain
\begin{eqnarray}
\sum_{k=1}^M c_k L_ke^{2 \pi i \beta_k x}=0.
\end{eqnarray}
Since $\mid \mathbb{R} \setminus E \mid = 0$, then equality (3.1) holds almost everywhere, and so Proposition \ref{trigo} and the fact that  $L_k\neq 0$ lead to a contradiction.

{\it (b)} If $|l_g(1) | =a \neq 0$,  then, by Proposition 3.3, $|l_g(\alpha)| = a^\alpha$, and, in particular, $l_g(\alpha) \neq 0$ for each $\alpha \in \mathbb{R}$.
Using Proposition 1.3 and the fact that the set $\Lambda$ satisfies the difference condition for the second variable, we suppose that
\begin{eqnarray}
g(x)= \sum_{k=1}^{N} c_k e^{2 \pi i \beta_k x}g(x+ \alpha_k)&
\mbox{a.e.}, \nonumber
\end{eqnarray}
where $c_1,\ldots,c_{N} \in \mathbb{C}$, $\alpha_1, \ldots, \alpha_{N} \in  \mathbb{R}$, and
$\beta_1, \ldots, \beta_{N} \in  \mathbb{R} \setminus \{0\}$.
Let $\{x_n\}$ be a positive sequence converging to infinity, with the property that the sequence $\{e^{2 \pi i \beta_k x_n}\}$ converges to a limit $L_k$ for each $k \in \{ 1, \ldots, N\}$. Proceeding as in  case {\it (a)}, we obtain
\begin{eqnarray}
\sum_{k=1}^{N-1} c_k l_g(\alpha_k)L_ke^{2 \pi i \beta_k x}= 1 \quad \text{ a.e.} \nonumber
\end{eqnarray}
Proposition \ref{trigo} and the facts that $l_g(\alpha_k) \neq 0$, $L_k \neq 0$, $\beta_k \neq 0$ for each $k \in \{ 1, \ldots, N\}$, and
  $c_k \neq 0$ for at least one $k \in \{ 1, \ldots, N\}$ lead to a contradiction.
\end{proof}

\begin{corollary}
Let $g \in L^2(\mathbb{R}) \setminus \{0\}$ and let $\Lambda \subseteq \mathbb{R}^2$ have the property that $\text{card}(\Lambda) \leq 5$. If $g$ and $\widehat{g}$ have ratio limits at every $\alpha \in \mathbb{R}$, then the HRT conjecture holds for $\mathcal{G}(g, \Lambda)$.
\end{corollary}

\begin{proof}
Suppose that $g$ and $\widehat{g}$ have ratio limits at
every $\alpha \in \mathbb{R}$.

If card$(\Lambda) \leq 3$, then the result is a consequence of known results, see Section 1.

Let card$(\Lambda)= 4$. By using the Fourier transform and the previous case, the only case which cannot follow by Theorem 3.9 is when
$$\Lambda =\{(\alpha_1,\beta_1), (\alpha_1,\beta_2),(\alpha_2,\beta_1),(\alpha_2,\beta_2) \}.$$
Hence, $\Lambda$ lies in a lattice and the HRT conjecture  holds for $\mathcal{G}(g, \Lambda)$ by known results, see Section 1.

Let card$(\Lambda)= 5$. Either  $\mathcal{G}(g,\Lambda)$ or $\mathcal{G}(\widehat{g},\widehat{\Lambda})$ satisfies the second difference condition,
where $\widehat{\Lambda}=\{(\beta,-\alpha): (\alpha,\beta) \in \Lambda \}$, and so we can apply Theorem 3.9 after using the previous cases.
\end{proof}

\begin{corollary} Let $E$ be a real vector space of real-valued functions having their germs in a Hardy field $\mathcal{H}$. Let $\Lambda$ be a set of finitely
many distinct points in $\mathbb{R}^2$ satisfying the difference condition for the second variable. The HRT conjecture holds for
$\mathcal{G}(g,\Lambda)$ if $g \thicksim h $, where $h$ is a finite linear combination (with complex coefficients) of functions in $E$.
\end{corollary}
\begin{proof} It suffices to notice that every finite linear combination of functions in $E$ is either half-line supported or has a ratio-limit at each positive number.
\end{proof}
Unlike Theorem \ref{thm:hrthardyfield}, $g$ does not need to be ultimately analytic and $\mathcal{H}$ is not required to be closed under translations in the case of Corollary 3.11.
\section{The HRT Conjecture for Functions with Exponential Decay}
\begin{lemma}
Let $\alpha  > 0$, let $M\geq 2$, and let $\beta_1,...,\beta_M \in \mathbb{R}$. For each $n \in \{1,...,M\}$, we define
\begin{eqnarray*}
\forall m=0,...,n-1,  \quad B_n(m) &=& \sum_{M-n+1\leq t_1<...<t_{n-m} \leq M}  e^{2 \pi i( b_{t_1} +b_{t_2}...+b_{t_{n-m})}\alpha}.
\end{eqnarray*}
Then, for each $n \in \{1,...,M-1\}$, we have
\begin{eqnarray*}
&&\forall m=1,...,n-1, \quad B_{n+1}(m)=e^{2\pi i b_{M-n}\alpha}B_n(m)+ B_n(m-1),\\
&&B_{n+1}(0)= e^{2\pi i b_{M-n}\alpha}B_n(0),  \mbox{ and }  B_{n+1}(n)= B_n(n-1)+e^{2\pi i b_{M-n}\alpha}.
\end{eqnarray*}
\end{lemma}
\begin{theorem}
Let $g$ be a measurable function on $\mathbb{R}$ such that  $e^{tx}g(x) \in L^1(\mathbb{R}) \setminus \{0\}$ for all $t>0$. Let $\Lambda=\{(\alpha_k, \beta_k)\}_{k=1}^N \subseteq \mathbb{R}^2$.
\begin{enumerate}[(a)]
  \item If,  for some $ k_0 \in \{1,......., N\}$, we have  $\alpha_k>\alpha_{k_0}$  for each $ k \in \{1,......., N\} \setminus \{k_0\}$, then the HRT conjecture holds for $\mathcal{G}(g, \Lambda)$.
  \item If $|g|$ is  ultimately decreasing, then the HRT conjecture holds for $\mathcal{G}(g, \Lambda)$, where $\Lambda$ is arbitrary.
\end{enumerate}
\end{theorem}
\begin{proof} \emph{(a)} Suppose that the HRT conjecture fails for  $\mathcal{G}(g,\{(\alpha_k, \beta_k)\}_{k=0}^N)$. If $\alpha_k>\alpha_0$  for each $ k \in \{1,......., N\}$, we use Proposition 1.3 to assume,  without loss of generality, that $(\alpha_0, \beta_0)=(0,0)$, and so we can write
$$g(x)= \sum_{k=1}^N c_k e^{2 \pi i \beta_k x}g(x+\alpha_k)\quad \mbox{a.e.},$$
where $c_1,...,c_N \in \mathbb{C} $, $\beta_1,...,\beta_N \in \mathbb{R}$, and
$\alpha_1,...,\alpha_N >0$. Therefore,
\begin{eqnarray*}
\forall t>0,  \quad \int g(x) e^{tx} dx \leq \sum_{k=1}^N |c_k|e^{-t\alpha_k} \int g(x) e^{tx} dx,
\end{eqnarray*}
and so
\begin{eqnarray*}
\forall t>0, \quad  1 \leq \sum_{k=1}^N |c_k|e^{-t\alpha_k}.
\end{eqnarray*}
Letting $t$ tend to $\infty$ in the last inequality leads to the desired contradiction.

\emph{(b) i.} If the HRT conjecture fails for  $\mathcal{G}(g,\Lambda)$, for some finite subset $\Lambda \subset \mathbb{R}^2$, we use Proposition 1.3 to assume,  without loss of generality, that
\begin{eqnarray*}
\sum_{m=1}^M a_m e^{2 \pi i b_m x}g(x)= G(x)\quad \mbox{a.e.}, \ \  \mbox{ where } \ \  G(x)= \sum_{k=1}^N c_k e^{2 \pi i \beta_k x}g(x+\alpha_k),
\end{eqnarray*}
$a_1,...,a_M,c_1,...,c_N \in \mathbb{C} \setminus \{0\}$, $b_1, ...,b_M$ are distinct real numbers, $\beta_1,...,\beta_N \in \mathbb{R}$, and
$\alpha_1,...,\alpha_N> 0$. Since \emph{(a)} deals with the case $M=1$, we assume that $M \geq 2$.

Since $|g|$ is ultimately decreasing, then either $g$ is supported on a half-line (in which case the HRT conjecture is satisfied, see Section 1) or $g$ is ultimately nonzero. Thus, without loss of generality, we suppose that $g$ is ultimately nonzero; and for the sake of simplicity, we assume that $g$ never vanishes.

\emph{ii.} Let $\alpha  > 0$. Then, we have
\begin{eqnarray*}
\sum_{m=1}^M a_m e^{2 \pi i b_m (x+\alpha)}g(x+\alpha)&=& G(x+\alpha) \quad \mbox{a.e.}
\end{eqnarray*}
Therefore, we compute
\begin{eqnarray*}
&&g(x+\alpha)e^{2 \pi i b_M \alpha}\sum_{m=1}^M a_m e^{2 \pi i b_m x}g(x)- g(x)\sum_{m=1}^M a_m e^{2 \pi i b_m (x+\alpha)}g(x+\alpha)  \\
&=&g(x+\alpha)G(x)e^{2 \pi i b_M \alpha} - g(x)G(x+\alpha) \quad \mbox{a.e.},
\end{eqnarray*}
and so
\begin{eqnarray*}
\sum_{m=1}^{M-1} a_m (e^{2 \pi i b_M \alpha} -e^{2 \pi i b_m \alpha}) e^{2 \pi i b_m x}g(x+\alpha)\\= \frac{g(x+\alpha)}{g(x)}G(x)e^{2 \pi i b_M \alpha} - G(x+\alpha) \quad \mbox{a.e.}
\end{eqnarray*}

After iterating the above process three times, we obtain
\begin{eqnarray*}
&&\sum_{m=1}^{M-3} a_m \prod_{l=M-2}^{M} (e^{2 \pi i b_l \alpha} -e^{2 \pi i b_m \alpha}) e^{2 \pi i b_m
x}g(x+n\alpha)\\&=&\frac{g(x+3\alpha)}{g(x)}G(x)e^{2 \pi i (b_M+b_{M-1} +b_{M-2}) \alpha}- \frac{g(x+3\alpha)}{g(x+\alpha)}G(x+\alpha)[\\
&&e^{2 \pi i (b_M+b_{M-1} ) \alpha}+e^{2 \pi i (b_M+b_{M-2} ) \alpha}+e^{2 \pi i (b_{M-1}+b_{M-2} ) \alpha}]\\&+& \frac{g(x+3\alpha)}{g(x+2\alpha)}G(x+2\alpha)[e^{2 \pi i b_M \alpha}+e^{2 \pi i b_{M-1} \alpha}+
e^{2 \pi i b_{M-2}  \alpha}]\\ &-& G(x+3\alpha) \quad \mbox{a.e.}
\end{eqnarray*}

Now, we invoke Lemma 4.1 to prove by induction on $n$ that the equality
\begin{eqnarray}
\sum_{m=1}^{M-n} a_m \prod_{l=M-n+1}^{M} (e^{2 \pi i b_l \alpha} -e^{2 \pi i b_m \alpha}) e^{2 \pi i b_m x}g(x+n\alpha) \\=
\sum_{m=0}^{n-1} (-1)^m B_n(m)\frac{g(x+n\alpha)}{g(x+m\alpha)}G(x+m\alpha)+(-1)^n G(x+n\alpha) \quad \mbox{a.e.} \nonumber
\end{eqnarray}
holds for each $n \in \{1,...,M-1\}$.

Writing equality (4.1) for $x+\alpha$ yields the equality
\begin{eqnarray}
&&\sum_{m=1}^{M-n} a_m \prod_{l=M-n+1}^{M} (e^{2 \pi i b_l \alpha} -e^{2 \pi i b_m \alpha}) e^{2 \pi i b_m (x+\alpha)}g(x+(n+1)\alpha)\\
&=&\sum_{m=0}^{n-1} (-1)^mB_n(m)\frac{g(x+(n+1)\alpha)}{g(x+(m+1)\alpha)}G(x+(m+1)\alpha)+(-1)^n G(x+(n+1)\alpha)\nonumber\\
&=&\sum_{m=1}^{n} (-1)^{m-1}B_n(m-1)\frac{g(x+(n+1)\alpha)}{g(x+m\alpha)}G(x+m\alpha)\nonumber \\&+&(-1)^n G(x+(n+1)\alpha) \quad \mbox{a.e.} \nonumber
\end{eqnarray}
Meanwhile, multiplying the two sides of equality (4.1) by  $e^{2\pi i b_{M-n}\alpha}g(x+(n+1)\alpha)/g(x+n\alpha)$ yields the equality
\begin{eqnarray*}
&&\sum_{m=1}^{M-n} a_m \prod_{l=M-n+1}^{M} (e^{2 \pi i b_l \alpha} -e^{2 \pi i b_m \alpha})e^{2\pi i b_{M-n}\alpha} e^{2 \pi i b_m x}g(x+(n+1)\alpha)\\ &&=
\sum_{m=0}^{n-1} e^{2\pi i b_{M-n}\alpha}(-1)^m B_n(m)\frac{g(x+(n+1)\alpha)}{g(x+m\alpha)}G(x+m\alpha)\\
&&+(-1)^n e^{2\pi i b_{M-n}\alpha}\frac{g(x+(n+1)\alpha)}{g(x+n\alpha)} G(x+n\alpha) \quad \mbox{a.e.} \nonumber
\end{eqnarray*}
Therefore, subtracting equality (4.2) from the last equality, we obtain
\begin{eqnarray*}
&&\sum_{m=1}^{M-n-1} a_m \prod_{l=M-n}^{M} (e^{2 \pi i b_l \alpha} -e^{2 \pi i b_m \alpha}) e^{2 \pi i b_m x}g(x+(n+1)\alpha)\\
&&= e^{2\pi i b_{M-n}\alpha}B_n(0)\frac{g(x+(n+1)\alpha)}{g(x)}G(x) \nonumber \\
&&+\sum_{m=1}^{n-1}(-1)^m \left[e^{2\pi i b_{M-n}\alpha}B_n(m)+ B_n(m-1)\right] \frac{g(x+(n+1)\alpha)}{g(x+m\alpha)}G(x+m\alpha) \nonumber \\
&&+ (-1)^{n}\left[B_n(n-1)+e^{2\pi i b_{M-n}\alpha} \right]\frac{g(x+(n+1)\alpha)}{g(x+n\alpha)}G(x+n\alpha) \nonumber \\
&&-(-1)^{n} G(x+(n+1)\alpha)\quad \mbox{a.e.},\nonumber
\end{eqnarray*}
and so, using Lemma 4.1, we conclude that
\begin{eqnarray*}
\sum_{m=1}^{M-n-1} a_m \prod_{l=M-n}^{M} (e^{2 \pi i b_l \alpha} -e^{2 \pi i b_m \alpha}) e^{2 \pi i b_m x}g(x+(n+1)\alpha)\\
=\sum_{m=0}^{n} (-1)^m B_{n+1}(m)\frac{g(x+(n+1)\alpha)}{g(x+m\alpha)}G(x+m\alpha)+(-1)^{n+1} G(x+(n+1)\alpha)\quad \mbox{a.e.}; \nonumber
\end{eqnarray*}
and this completes the induction proof.

\emph{iii.} Writing (4.1) for $M-1$ yields the equality
\begin{eqnarray*}
a_1 \prod_{m=2}^{M} (e^{2 \pi i b_m \alpha} -e^{2 \pi i b_1 \alpha}) e^{2 \pi i b_1 x}g(x+(M-1)\alpha)\\=  \sum_{m=0}^{M-1} B(m)\frac{g(x+(M-1)\alpha)}{g(x+m\alpha)}G(x+m\alpha) \quad \mbox{a.e.},
\end{eqnarray*}
where, for $0\leq m \leq M-2$, $B(m)=(-1)^m B_{M-1}(m)$ and $B(M-1)=(-1)^{M-1}$.

Let $t>0$. Since $\mid g \mid $ is ultimately decreasing, there is $A \in \mathbb{R}$ for which \\ $|g(x+(M-1)\alpha)/g(x+m\alpha) | < 1$, for $0\leq m\leq M-1$ and for all $x>A$. Therefore,
\begin{eqnarray*}
&& |a_1\prod_{m=2}^{M} (e^{2 \pi i b_m \alpha} -e^{2 \pi i b_1 \alpha}) | \int_{A}^\infty  \mid g(x+(M-1)\alpha) | e^{tx} dx\\
&&\leq \sum_{m=0}^{M-1} | B(m)| \int_{A}^\infty | G(x+m\alpha)| e^{tx} dx.
\end{eqnarray*}

By the definition of $G$, we compute
\begin{eqnarray*}
&& | a_1\prod_{m=2}^{M} (e^{2 \pi i b_m \alpha} -e^{2 \pi i b_1 \alpha}) | \int_{A}^\infty  | g(x+(M-1)\alpha) | e^{tx} dx\\
&&\leq \sum_{m=0}^{M-1} | B(m)\mid  \sum_{k=1}^{N} \mid c_k | \int_{A}^\infty | g(x+\alpha_k+m\alpha)| e^{tx} dx,
\end{eqnarray*}
and so
\begin{eqnarray*}
&& | a_1\prod_{m=2}^{M} (e^{2 \pi i b_m \alpha} -e^{2 \pi i b_1 \alpha}) | \int_{A}^\infty  | g(x+(M-1)\alpha) | e^{tx} dx\\
&&\leq \sum_{m=0}^{M-1} | B(m)\mid  \sum_{k=1}^{N} | c_k | \int_{A}^\infty | g(x+(M-1)\alpha)| e^{t(x-\alpha_k+(M-1-m)\alpha)} dx.
\end{eqnarray*}

Choosing $\alpha$ such that $0< (M-1)\alpha < \inf \{ \alpha_1,...,\alpha_N\}$, we can write
\begin{eqnarray}
&& \mid a_1\prod_{m=2}^{M} (e^{2 \pi i b_m \alpha} -e^{2 \pi i b_1 \alpha}) \mid \int_{A}^\infty  \mid g(x+(M-1)\alpha) \mid e^{tx} dx\\
&&\leq \sum_{m=0}^{M-1} \mid B(m)\mid  \sum_{k=1}^{N} \mid c_k \mid e^{-t(\alpha_k -(M-1-m)\alpha)} \int_{A}^\infty \mid g(x+(M-1)\alpha)\mid e^{tx} dx.\nonumber
\end{eqnarray}
Then, using the fact that $\mid g \mid $ is ultimately positive, we can simplify (4.3) and obtain
\begin{eqnarray*}
\mid a_1\prod_{m=2}^{M} (e^{2 \pi i b_m \alpha} -e^{2 \pi i b_1 \alpha}) \mid
\leq \sum_{m=0}^{M-1} \mid B(m)\mid  \sum_{k=1}^{N} \mid c_k \mid e^{-t(\alpha_k -(M-1-m)\alpha)}.
\end{eqnarray*}

Letting $t$  tend $\infty$ in the last inequality yields the contradiction
\begin{eqnarray*}
\forall 0<\alpha  < \frac{\inf \{ \alpha_1,...,\alpha_N\}}{M-1}, \quad \mid a_1\prod_{m=2}^{M} (e^{2 \pi i b_m \alpha} -e^{2 \pi i b_1 \alpha}) \mid \leq 0.
\end{eqnarray*}
\end{proof}

\begin{remark}\emph{(a)} Let $A \in \mathbb{R}$. Theorem 4.1 remains true if we replace its first assumption with the assumption that $g$ is a measurable function on $\mathbb{R}$ such that  $e^{tx}g(x) \in L^1([A,\infty))$ for all $t>0$.

\emph{(b)} Let $p>1$. Theorem 4.1 remains true if we replace its first assumption with the assumption that $g$ is a measurable function on $\mathbb{R}$ such that  $e^{tx}g \in L^p(\mathbb{R})$ for all $t>0$.

\emph{(c)} Theorem 4.1 stays true if we replace its first assumption with the assumption that $g$ is a measurable function on $\mathbb{R}$ such that  $\lim_{x \rightarrow \infty} e^{tx}g(x) =0$, for all $t>0$. This result was recently obtained independently in ~\cite{Bow2} by using different techniques.

\emph{(d)} Statement \emph{(b)} of Theorem 4.1 remains true if we replace the assumption that $g$ is ultimately decreasing  with the weaker assumption that for each $a >0$, $|g(x+a)/g(x)|$ is ultimately bounded. This is the case if $|g|$ ultimately has a bounded logarithmic derivative.
\end{remark}

\section{The HRT Conjecture for Positive Functions}
For this section we require the following result, see ~\cite{Ben} Section 3.2.12, ~\cite{Har1}  Chapter XXIII, ~\cite{Kat} Chapter VI.9, ~\cite{Shi}.
\begin{theorem}[Kronecker's Approximation Theorem]\label{thm:kronecker}
Let $\{\beta_1,\ldots, \beta_N\} \subseteq \mathbb{R}$ be a linearly independent set over $\mathbb{Q}$, and let
$\theta_1,\ldots, \theta_N \in \mathbb{R}$. If $U, \varepsilon >0$, then there exist $p_1, \ldots, p_N \in \mathbb{Z}$ and $u >U$ such that
$$\forall k = 1, \ldots, N, \quad |\beta_ku - p_k - \theta_k| < \varepsilon,$$
and, therefore,
$$\forall k = 1, \ldots, N, \quad |e^{2 \pi i \beta_k u}- e^{2 \pi i \theta_k}| < 4 \pi \varepsilon.$$
\end{theorem}

\begin{theorem}Let $g \in L^2(\mathbb{R})$ and assume that $g$ is ultimately positive.
Let $\Lambda = \{(\alpha_k, \beta_k)\}_{k=0}^N \subseteq \mathbb{R}^2$ have the property that $\{\beta_0,\ldots,\beta_N\}$
is linearly independent over $\mathbb{Q}$. The HRT conjecture holds for $\mathcal{G}(g,\Lambda)$.
\end{theorem}

\begin{proof}
If $\{\beta_0, . . ., \beta_N\}$  is linearly independent over $\mathbb{Q}$, then $\{\beta_1-\beta_0, . . ., \beta_N-\beta_0\}$ is also linearly independent over $\mathbb{Q}$. Using Proposition 1.3, we assume that $(\alpha_0, \beta_0) = (0,0)$, and so $\{\beta_1, . . ., \beta_N\}$ is linearly independent over $\mathbb{Q}$. Assuming that $\mathcal{G}(g, \Lambda)$ is linearly dependent in $L^2(\mathbb{R})$, we shall obtain a contradiction.

{\it i.} The linear dependence of $\mathcal{G}(g,\Lambda)$ implies, without loss of generality, that there are $c_1,\ldots, c_N \in \mathbb{C} \setminus \{0\}$ such that
\begin{eqnarray}
  g(x) &=& \sum_{k=1}^N  c_k e^{2 \pi i \beta_kx}g(x-\alpha_k) \quad \text{a.e.}
\end{eqnarray}

{\it ii.} By Kronecker's  theorem (Theorem \ref{thm:kronecker}) and the linear independence of $\{\beta_1, \ldots, \beta_N\} \subseteq \mathbb{R}$ over $\mathbb{Q}$, there exists a sequence $\{u_n\} \subseteq \mathbb{R}$ such that $\lim_{n \rightarrow \infty} u_n = \infty$, and
\begin{equation}
\forall k = 1, \ldots, N, \quad \lim_{n \rightarrow \infty} e^{2 \pi i \beta_k u_n} = e^{2 \pi i \theta_k},
\end{equation}
where each
$$\theta_k = \phi_k + 1/4 \quad \text{and} \quad c_k = |c_k|e^{-2 \pi i \phi_k},$$
i.e., we have chosen $\theta_k$ in our application of Theorem 5.1 to be defined by the formula, $e^{2 \pi i \theta_k} = |c_k|i/c_k$. Therefore, from (5.2), we compute
\begin{equation*}
\forall k = 1, \ldots, N, \quad \lim_{n \rightarrow \infty} c_k e^{2 \pi i \beta_k u_n} = |c_k|i.
\end{equation*}

By Lemma 3.8, there is a set $X \subseteq \mathbb{R}$, $|\mathbb{R} \setminus X| = 0$, such that
\begin{equation}
 \forall (n,x) \in \mathbb{N} \times X, \quad g(x+u_n) = \sum_{k=1}^N  c_k e^{2 \pi i \beta_k(x+n)}g(x+u_n-\alpha_k).
\end{equation}

{\it iii.} For the sake of simplicity, we assume that $0 \in X$. Since $g$ is ultimately positive and $u_n \rightarrow \infty$ we can assume that
$$\forall n \text{ and } \forall \, k = 0, \ldots, N, \quad g(u_n - \alpha_k) > 0.$$
Because of the positivity, we use (5.3) with $x=0$  to write
\begin{align}
1 &= \sum_{k = 1}^N \left( |c_k| i + \left(c_k e^{2 \pi i \beta_k u_n} - |c_k|i \right) \right) \frac{g(u_n - \alpha_k)}{g(u_n)} \\
& \geq \left| \sum_{k=1}^N |c_k| i \frac{g(u_n - \alpha_k)}{g(u_n)} \right| - \left| \sum_{k=1}^N \left(c_k e^{2 \pi i \beta_k u_n} - |c_k|i \right) \frac{g(u_n - \alpha_k)}{g(u_n)}\right| \nonumber \\
& \geq \sum_{k=1}^N |c_k| \frac{g(u_n - \alpha_k)}{g(u_n)} - \sum_{k=1}^N |c_k| \left| e^{2 \pi i \beta_k u_n}- \frac{|c_k|}{c_k}i\right|\frac{g(u_n- \alpha_k)}{g(u_n)}, \nonumber
\end{align}
since $|cd - |c|i| = |c||d - |c|i/c|$ for $c \in \mathbb{C}\setminus \{0\}$ and $d \in \mathbb{C}$.

Let $\varepsilon = 1/2$ in Theorem \ref{thm:kronecker}. Then, we have that
$$\exists \, U > 0 \text{ such that } \forall \, u_n > U \text{ and } \forall \, k = 1, \ldots, N, $$
$$\left| e^{2 \pi i \beta_k u_n} - \frac{|c_k|}{c_k}i\right| < \frac{1}{2}.$$
Consequently, (5.4) allows us to assert that
\begin{equation*}
\forall \, u_n > U, \quad 2 \geq \sum_{k = 1}^N |c_k| \frac{g(u_n - \alpha_k)}{g(u_n)} ;
\end{equation*}
and, hence, the sequence
$$\left\{ \frac{g(u_n- \alpha_k)}{g(u_n)}\right\}$$
is bounded for each $k = 1, \ldots, N$. Therefore, there is a subsequence $\{v_n\}$ of $\{u_n\}$ for which $\{g(v_n - \alpha_k)/g(v_n)\}$ converges to some $r_k \in \mathbb{R}$ for each $k = 1, \ldots, N$. Consequently, by invoking Theorem \ref{thm:kronecker} again, and replacing $u_n$ by $v_n$, the equality in (5.4) leads to
$$1 = \sum_{k=1}^N |c_k| r_k i, $$
the desired contradiction.
\end{proof}
\begin{lemma} Let $g \in L^2(\mathbb{R})$ have the properties that $g(x)$ and $g(-x)$ are ultimately  positive and ultimately decreasing.
Define
 \begin{eqnarray*}
   \Delta_{jk}(x,y) &=&  g(x+y+\alpha_j)g(x-y+\alpha_k)- g(x-y+\alpha_j)g(x+y+\alpha_k),
\end{eqnarray*}
where $x, y, \alpha_j,\alpha_k \in \mathbb{R}$. Assume that $\alpha_j < \alpha_k$ and let $x \in \mathbb{R}$.
\begin{enumerate}[(a)]
  \item If $y$ is large enough, then $\Delta_{jk}(x,y) \geq 0$.
  \item If  $\Delta_{jk}(x,y) = 0$ and $y$ is large enough, then $g(x+y+\alpha_j)= g(x+y+\alpha_k)$ and $g(x-y+\alpha_j)= g(x-y+\alpha_k)$.
\end{enumerate}
\end{lemma}
\begin{lemma} Let $(\beta_1, \beta_2, \beta_3) \in \mathbb{R} \setminus \{0\}$, let $c_1, c_2, c_3 \in \mathbb{C} \setminus \{0\}$, and let $E,F \subseteq \mathbb{R}$ have the properties that $|E|,|F|>0$. If
\begin{eqnarray*}
\forall x \in E,&  &c_1 e^{2 \pi i\beta_1x}  +c_2e^{2 \pi i\beta_2x} + c_3e^{2 \pi i\beta_3x} \in \mathbb{R}
\end{eqnarray*}
and
\begin{eqnarray*}
\forall x \in F,&  &\frac{c_1}{c_3} e^{2 \pi i(\beta_1-\beta_3)x}+\frac{c_2}{c_3} e^{2 \pi i(\beta_2-\beta_3)x} -\frac{1}{c_3} e^{-2 \pi i\beta_3x} \in \mathbb{R},
\end{eqnarray*}
then one of the following statements is satisfied.
\begin{enumerate}[(a)]
  \item $\beta_3=0$ and $\beta_1= \beta_2 \neq 0$;  and, in this case, we have $c_3 \in \mathbb{R}$ and $c_1 +c_2=0$. 
  \item $\beta_3=0$ and $\beta_2= -\beta_1 \neq 0$;  and, in this case, we have $c_3 \in \mathbb{R}$ and $c_2 =\overline{c_1}$. 
\end{enumerate}
\end{lemma}

\begin{lemma} Let $g \in L^2(\mathbb{R})$ , let $(\beta_1, \beta_2, \beta_3) \in \mathbb{R} \setminus \{0\}$,  and  let $0<\alpha_1< \alpha_2<\alpha_3 $. If there is $a \in \mathbb{R}$ for which we have $g$  positive on $  [a-\alpha_1, a +2\alpha_3-\alpha_2]$ and constant on $[a, a+\alpha_3]$,   then the  HRT conjecture holds for $\mathcal{G}(g, \{(-\alpha_k,\beta_k)\}_{k=0}^3)$, where $(\alpha_0,\beta_0)=(0,0)$.
\end{lemma}

\begin{proof} {\it i.} Suppose that $\mathcal{G}(g, \{(-\alpha_k,\beta_k)\}_{k=0}^3)$ is a linearly dependent set of functions. Recall (Section 1) that the HRT conjecture holds for any three point set. We use this fact two times, combined with the assumption of linear dependence and a straightforward calculation to show that  there are $c_1, c_2, c_3 \in \mathbb{C} \setminus \{0\}$ such that
\begin{eqnarray}
 g(x)=  \sum_{k=1}^3 c_ke^{2 \pi i \beta_kx}g(x+\alpha_k) \quad \text{a.e.}
\end{eqnarray}
Therefore, the fact that g is constant on $[a, a+\alpha_3]$ implies that
\begin{eqnarray*}
\forall x \in E, \quad   g(x)&=&[c_1e^{2 \pi i\beta_1x} +c_2e^{2 \pi i\beta_2x}+c_3e^{2 \pi i\beta_3x}]g(a)
\end{eqnarray*}
and
\begin{eqnarray*}
\forall x \in F, \quad   g(x)&=& [-\frac{c_1}{c_3} e^{2 \pi i(\beta_1-\beta_3)x}-\frac{c_2}{c_3} e^{2 \pi i(\beta_2-\beta_3)x} + \frac{1}{c_3} e^{-2 \pi i\beta_3x}]g(a),
\end{eqnarray*}
where $  E = [a-\alpha_1, a]$ and $  F = [a +\alpha_3, a +2\alpha_3-\alpha_2]$. Hence,  the fact that $g$ is positive on $  [a-\alpha_1, a +2\alpha_3-\alpha_2]$  implies that
\begin{eqnarray*}
\forall x \in E,&  &c_1 e^{2 \pi i\beta_1x}  +c_2e^{2 \pi i\beta_2x} + c_3e^{2 \pi i\beta_3x} \in \mathbb{R}
\end{eqnarray*}
and
\begin{eqnarray*}
\forall x \in F,&  &\frac{c_1}{c_3} e^{2 \pi i(\beta_1-\beta_3)x}+\frac{c_2}{c_3} e^{2 \pi i(\beta_2-\beta_3)x} -\frac{1}{c_3} e^{-2 \pi i\beta_3x} \in \mathbb{R}.
\end{eqnarray*}
Consequently, Lemma 5.4 lists all the possible cases relating $\beta_1, \beta_2,$ and $\beta_3$. In part \emph{ii},  we shall see that each one of these cases leads to a contradiction.

{\it ii.}  Assume that $\beta_3=0$ and $\beta_1= \beta_2 \neq 0$. In this case, we have $c_3 \in \mathbb{R}$ and $c_1 +c_2=0$. Thus, (5.5) is
\begin{eqnarray*}
 g(x)=  c_1e^{2 \pi i \beta_1x}[g(x+\alpha_1)-g(x+\alpha_2)]+c_3g(x+\alpha_3),
\end{eqnarray*}
and so $ \{ x: g(x+\alpha_1) \neq g(x+\alpha_2) \}  \subseteq \{ x: c_2e^{2 \pi i \beta_1x} \in \mathbb{R}  \}$. Meanwhile, the fact that
$g \in L^2(\mathbb{R})$ implies that $|\{ x: g(x) \neq g(x+\alpha_1) \}| \neq 0$. Therefore, we obtain  the contradiction $|\{ x: c_2e^{2 \pi i \beta_2x} \in \mathbb{R}  \}|\neq 0$.

Similarly, we obtain the desired contradiction for the case where $\beta_3=0$ and $\beta_2= -\beta_1 \neq 0$.


\end{proof}

\begin{theorem}Let $g \in L^2(\mathbb{R})$ have the properties that $g(x)$ and $g(-x)$ are ultimately  positive and ultimately decreasing, and let $\Lambda = \{(\alpha_k,\beta_k)\}_{k=0}^3 \subseteq \mathbb{R}^2$. The HRT conjecture holds for $\mathcal{G}(g,\Lambda)$.
\end{theorem}

\begin{proof} {\it i.} Suppose that $\mathcal{G}(g,\Lambda)$ is a linearly dependent set of functions, where $\Lambda = \{(-\alpha_k,\beta_k)\}_{k=0}^3 \subseteq \mathbb{R}^2$. Using Proposition 1.3, we assume, without loss of generality, that $(\alpha_0, \beta_0)=(0,0)$ and $0\leq \alpha_1 \leq  \alpha_1 \leq \alpha_3$. Since The HRT conjecture holds for any three point set,  there are $c_1, c_2, c_3 \in \mathbb{C} \setminus \{0\}$
such that
\begin{eqnarray}
 g(x)=  \sum_{k=1}^3 c_ke^{2 \pi i \beta_kx}g(x+\alpha_k) \quad \text{a.e.}
\end{eqnarray}

{\it ii.} Since $g(x)$ and  $g(-x)$ are  positive and decreasing on $(a,\infty)$ for some $a>0$,  then $|\{ x: |x|>a \mbox{ and } g \mbox{ is discontinuous at } x \}| =0$  and the left hand limit $g(x^-)$ exists at each $x$ for which $|x|>a$. Therefore, if $h(x)=g(x^-)$ for $|x|>a$ and $h(x)=g(x)$ elsewhere, then $h=g$ a.e., and so  we obtain
\begin{eqnarray*}
h(x)=  \sum_{k=1}^3 c_ke^{2 \pi i \beta_kx}h(x+\alpha_k) \quad \text{a.e.}
\end{eqnarray*}
Since $h$ is  left hand continuous  on $\{ x: |x|>a  \}$, the last equality holds for all $x$ for which $|x| >a$.

{\it iii.} For each of the remaining steps of our proof, it will suffice to assume that (5.6) holds for $|x|$ as large as we wish. Hence, for the sake of simplicity and without loss of generality, we assume that $g$ is  positive on $\mathbb{R}$, decreasing on $(0,\infty)$, increasing on $(-\infty,0)$, and that (5.6) holds everywhere. In particular, for each $n \in \mathbb{N}$, we have
\begin{equation}
 1= \sum_{k=1}^3 \frac{g(n+\alpha_k)}{g(n)}c_ke^{2 \pi i \beta_kn}
  \end{equation}
  and
  \begin{eqnarray}
  \frac{g(-n)}{g(-n+\alpha_3)} &=& \sum_{k=1}^3 \frac{g(-n+\alpha_k)}{g(-n+\alpha_3)}c_ke^{-2 \pi i \beta_kn}.
  \end{eqnarray}
Using the hypothesis that $g$ is decreasing on $(0,\infty)$ and increasing on $(-\infty,0)$, we have that the sequences,
\begin{eqnarray*}
\{\frac{g(n+\alpha_k)}{g(n)}\}_{n>0}\ \ \ \mbox{ and } \ \ \ \{\frac{g(-n+\alpha_{k-1})}{g(-n+\alpha_3)}\}_{n>0},
\end{eqnarray*}
are bounded for each $k \in \{1,2,3\}$.

With this backdrop, we now use Theorem \ref{thm:kronecker} to construct a sequence\\ $\{u_n\}_{n>0} \subseteq \mathbb{N}$, resp.,  $\{v_n\}_{n>0} \subseteq -\mathbb{N}$, for which the sequence
$\{e^{2 \pi i \beta_ku_n}\}_{n>0}$,  resp.,  $\{e^{2 \pi i\beta_k v_n}\}_{n>0}$, converges to $e^{2 \pi i\theta_k }$, resp.,
$e^{2 \pi i\theta_k' }$,  for each $k \in \{1,2,3\}$. The degree of freedom with which the limits $e^{2 \pi i\theta_k }$ and $e^{2 \pi i\theta_k' }$
are chosen, for each $k \in \{1,2,3\}$, will depend on the properties of the set $\{\beta_1, \beta_2, \beta_3\}$. Next, we extract from the sequence
$\{g(u_n+\alpha_k)/g(u_n)\}_{n>0}$, resp., $\{g(v_n+\alpha_{k-1})/g(v_n+\alpha_3)\}_{n>0}$, a subsequence that converges to some  $l_k \geq 0$, resp.,
$l_{k-1}' \geq 0$, for each $k \in \{1,2,3\}$. The limits $l_k$, resp., $l_{k-1}'$, will depend on the choice of $\theta_k$, resp., $\theta_k'$, for each $k \in \{1,2,3\}$. The properties of $g$ imply that $0\leq l_3 \leq l_2 \leq l_1$ and $0\leq l'_0 \leq l'_1 \leq l'_2 \leq 1$. Further, $l_1 > 0$, resp.,  $l'_2 > 0$, or (5.7), resp., (5.8), leads to a contradiction. Using all of this we obtain the desired contradiction for each of the possible cases relating $\beta_1, \beta_2,$ and $\beta_3$. These cases are dealt with in parts {\it iv--viii}.

Let $d_1,d_2,d_3 \in \mathbb{R}$ have the property that $c_k=|c_k| e^{2 \pi i d_k}$ for each $k \in \{ 1,2,3 \}$.

{\it iv.} If $\{ \beta_1,\beta_2,\beta_3 \}$ is linearly independent over $\mathbb{Q}$, the independence  of $\mathcal{G}(g,\Lambda)$ is a consequence of Theorem 5.2.

{\it v.} Assume that $\{\beta_1,\beta_2\}$ is linearly independent over $\mathbb{Q}$ and
$\beta_3= r_1\beta_1+r_2\beta_2$, where $r_1, r_2 \in \mathbb{Q}$.

Let $(\theta'_1, \theta'_2) \in \mathbb{R}^2$. Using Theorem \ref{thm:kronecker}, we can choose $\{v_n\}$ such that
\begin{eqnarray*}
  \lim_{n \rightarrow \infty} e^{2 \pi i \beta_k v_n} = e^{2 \pi i(\theta'_k-d_k)} \quad \mbox{for each } k\in \{1,2\}.
\end{eqnarray*}
Thus, the limit of (5.8) gives
 \begin{eqnarray}
  l'_0 &=& l'_1|c_1|e^{2 \pi i \theta'_1}+l'_2|c_2|e^{2 \pi i \theta'_2}+|c_3|e^{2 \pi i[r_1(\theta'_1-d_1) + r_2(\theta'_2-d_2) +d_3]},
\end{eqnarray}
where $l'_0, l'_1$ and $l'_2$ are nonnegative real numbers that depend on the choice of $(\theta'_1, \theta'_2)$ and $l'_2 > 0$.

For $(\theta'_1, \theta'_2)=(0,0)$, (5.9) is
\begin{eqnarray*}
  l'_0 &=& l'_1|c_1|+l'_2|c_2|+ |c_3|e^{2 \pi i[-r_1d_1-r_2d_2 +d_3]},
\end{eqnarray*}
and so $e^{2 \pi i[-r_1d_1-r_2d_2 +d_3]}= \epsilon \in \{-1,1\}$. Therefore, for an arbitrary $(\theta'_1, \theta'_2)$, (5.9) can be rewritten as
\begin{eqnarray}
  l'_0 &=& l'_1|c_1|e^{2 \pi i \theta'_1}+l'_2|c_2|e^{2 \pi i \theta'_2} +\epsilon |c_3|e^{2 \pi i [r_1\theta'_1 +r_2\theta'_2)]}.
\end{eqnarray}

For $(\theta'_1, \theta'_2)=(1/2,0)$, (5.10) is
\begin{eqnarray*}
  l'_0 &=& -l'_1|c_1|+l'_2|c_2| +\epsilon |c_3|e^{ \pi i r_1},
\end{eqnarray*}
and so $r_1 \in \mathbb{Z}$. Similarly, we can prove that $r_2 \in \mathbb{Z}$ by taking $(\theta'_1, \theta'_2)=(0,1/2)$.

For $(\theta'_1, \theta'_2)=(0,1/4)$, (5.10) is
\begin{eqnarray}
  l'_0 &=& l'_1|c_1|+l'_2|c_2|i +\epsilon |c_3|i^{r_2}.
\end{eqnarray}
Since $l'_2 > 0$, then $r_2$ must be an odd number and, in particular, $r_2 \neq 0$.

For $(\theta'_1, \theta'_2)=(0,1/r_2)$, (5.10) is
\begin{eqnarray*}
  l'_0 &=& l'_1|c_1|+l'_2|c_2|e^{2 \pi i/r_2} +\epsilon |c_3|.
\end{eqnarray*}
Since $l'_2 > 0$, then $r_2=1$, and so $\epsilon=-1$ by using $(5.11)$. Therefore, for an arbitrary $(\theta'_1, \theta'_2)$, (5.10) can be rewritten as
\begin{eqnarray*}
  l'_0 &=& l'_1|c_1|e^{2 \pi i \theta'_1}+l'_2|c_2|e^{2 \pi i \theta'_2} - |c_3|e^{2 \pi i [r_1\theta'_1 + \theta'_2)]}.
\end{eqnarray*}

For $(\theta'_1, \theta'_2)=(1/4,1/4)$, the last equality is
\begin{eqnarray*}
  l'_0 &=& l'_1|c_1|i+l'_2|c_2|i -|c_3|i^{r_1} i.
\end{eqnarray*}
Since $l'_2 > 0$, then $r_1=4p$ for some integer $p$.

Let $(\theta_1, \theta_2) \in \mathbb{R}^2$. Using Theorem \ref{thm:kronecker} again, we can choose $\{u_n\}$ such that
\begin{eqnarray*}
  \lim_{n \rightarrow \infty} e^{2 \pi i \beta_k u_n} = e^{2 \pi i(\theta_k-d_k)} \quad \mbox{for each } k\in \{1,2\}.
\end{eqnarray*}
Thus, the limit of (5.7) gives
 \begin{eqnarray}
  1 &=& l_1|c_1|e^{2 \pi i \theta_1}+l_2|c_2|e^{2 \pi i \theta_2}-l_3|c_3|e^{2 \pi i[4p\theta_1+\theta_2]},
\end{eqnarray}
where $l_1, l_2$ and $l_3$ are nonnegative real numbers that depend on the choice of $(\theta_1, \theta_2)$ and $l_1>0$.

For $(\theta_1, \theta_2)=(1/4,0)$, (5.12) is
\begin{eqnarray*}
  1 &=& l_1|c_1|i+l_2|c_2|-l_3|c_3|.
\end{eqnarray*}
The fact that $l_1 \neq 0$ and the last equality provide the desired contradiction.

{\it vi.} Assume that $\{ \beta_1,\beta_3\}$ is linearly independent over $\mathbb{Q}$ and $\beta_2= r\beta_1$, where $r \in \mathbb{Q}$.

Let $(\theta'_1, \theta'_3) \in \mathbb{R}^2$. Using Theorem \ref{thm:kronecker} once again and proceeding as in part \emph{iii}, we can choose  $\{v_n\}$ for which  the limit of (5.8) gives
 \begin{eqnarray}
  l'_0 &=& l'_1|c_1|e^{2 \pi i \theta'_1}+l'_2|c_2|e^{2 \pi i [r(\theta'_1-d_1)+d_2]}+|c_3|e^{2 \pi i\theta'_3}.
\end{eqnarray}

For $(\theta'_1, \theta'_3)= (0,0)$, (5.13) is
\begin{eqnarray*}
  l'_0 &=& l'_1|c_1|+l'_2|c_2|e^{2 \pi i [-rd_1+d_2]}+|c_3|.
\end{eqnarray*}
Since $l'_2 \neq 0$, we have $e^{2 \pi i [-rd_1 +d_2]}= \epsilon \in \{-1,1\}$, and so, for an arbitrary
$(\theta'_1, \theta'_3)$, (5.13) becomes
\begin{eqnarray*}
  l'_0 &=& l'_1|c_1|e^{2 \pi i \theta'_1} +\epsilon l'_2|c_2|e^{2 \pi i r\theta'_1}+|c_3|e^{2 \pi i\theta'_3}.
\end{eqnarray*}

For $(\theta'_1, \theta'_3)= (0,1/4)$, the last equality  is
\begin{eqnarray*}
  l'_0 &=& l'_1|c_1| + \epsilon l'_2|c_2|+|c_3|i,
\end{eqnarray*}
and this leads to the contradiction, $c_3=0$.

{\it vii.}  Assume that $\beta_1 = 0$  and $\{ \beta_2,\beta_3\}$ is linearly independent over $\mathbb{Q}$.

Let $(\theta_2, \theta_3),(\theta'_2, \theta'_3) \in \mathbb{R}^2$. By Theorem \ref{thm:kronecker}, we can choose $\{u_n\}$ and $\{v_n\}$ such that
\begin{eqnarray*}
 \lim_{n \rightarrow \infty} e^{2 \pi i \beta_k u_n} = e^{2 \pi i(\theta_k-d_k)} \ \ \mbox{ and }  \ \
 \lim_{n \rightarrow \infty} e^{2 \pi i \beta_k v_n} = e^{2 \pi i(\theta'_k-d_k)}
\end{eqnarray*}
for each  $k\in \{2,3\}$. Thus, equalities (5.7) and (5.8) become
 \begin{eqnarray}
 1 &=& l_1|c_1|e^{2 \pi i d_1}+l_2|c_2|e^{2 \pi i \theta_2}+l_3|c_3|e^{2 \pi i\theta_3},
\end{eqnarray}
where $l_1, l_2$ and $l_3$ are nonnegative real numbers that depend on the choice of $(\theta_1, \theta_2)$, $l_1 >0$, and
\begin{eqnarray}
  l'_0 &=& l'_1|c_1|e^{2 \pi i d_1}+l'_2|c_2|e^{2 \pi i \theta'_2}+|c_3|e^{2 \pi i\theta'_3},
\end{eqnarray}
where $l'_0, l'_1$ and $l'_2$ are nonnegative real numbers that depend on the choice of $(\theta'_1, \theta'_2)$ , and $l'_2 > 0$.

For $(\theta_2, \theta_3)= (0,0)$, (5.14) is
\begin{eqnarray*}
  1 &=& l_1|c_1|e^{2 \pi i d_1}+l_2|c_2|+l_3|c_3|.
\end{eqnarray*}
Since $l_1 > 0$, we have $e^{2 \pi id_1}=\pm 1$, and so, for an arbitrary $(\theta'_2, \theta'_3)$, (5.15) becomes
\begin{eqnarray*}
  l'_0 &=& \pm l'_1|c_1|+l'_2|c_2|e^{2 \pi i \theta'_2}+|c_3|e^{2 \pi i\theta'_3}.
\end{eqnarray*}

For $(\theta'_2, \theta'_3)= (0,1/4)$, the last equality  is
\begin{eqnarray*}
  l'_0 &=& \pm l'_1|c_1|+l'_2|c_2|+|c_3|i,
\end{eqnarray*}
and this leads to the contradiction, $c_3=0$.

{\it viii.}  Assume that  $(\beta_1, \beta_2,\beta_3)=(r_1\beta, r_2\beta,r_3\beta)$, where $\beta \in \mathbb{R}$ and $r_1\, r_2,r_3 \in \mathbb{Q}$.  We use Proposition 1.3 to assume, without loss of generality, that $\beta_1, \beta_2, \beta_3 \in \mathbb{Z}$.

{\it viii.a.} If $\beta_1= \beta_2=\beta_3=0$ or $\alpha_1=\alpha_2=\alpha_3=0$, then the set
$\Lambda = \{(\alpha_k,\beta_k)\}_{k=0}^3$ is a subset of a lattice and the linear independence of $\mathcal{G}(\Lambda,g)$ is a consequence of
known results, see Section 1.

{\it viii.b.} For the remaining subcases, we assume that $(\beta_1, \beta_2,\beta_3) \neq (0,0,0)$ and $\alpha_3>0$. Therefore,
\begin{eqnarray}
\exists n  \in  \mathbb{N}, \quad   |\{x: \Delta_{03}(x,n) >0 \}| \neq 0.
\end{eqnarray}

Indeed, if (5.16) does not hold, then there are $a,b \in \mathbb{R}$ for which we have $b - \alpha_3 <a < b$ and $\Delta_{03}(a,n) = \Delta_{03}(b,n)=0$. Therefore, taking $n$ large enough and using Lemma 5.3, we obtain that
\begin{eqnarray*}
\forall x \in [a, a +\alpha_3], \quad   g(x+n)= g(a+n)
\end{eqnarray*}
and
\begin{eqnarray*}
\forall x \in [b, b +\alpha_3], \quad   g(x+n)= g(b+n).
\end{eqnarray*}
Hence, (5.6) leads to the contradiction
\begin{eqnarray*}
\forall x \in [a,b], \quad   g(a+n) = [c_1e^{2 \pi i\beta_1x} +c_2e^{2 \pi i\beta_2x}+c_3e^{2 \pi i\beta_3x}]g(a+n),
\end{eqnarray*}
since,  by Proposition 1.1, $|\{x:1 = c_1e^{2 \pi i\beta_1x} +c_2e^{2 \pi i\beta_2x}+c_3e^{2 \pi i\beta_3x}\}|=0$.

We shall also invoke Lemma 5.3 and Lemma 5.5; and use the following equalities:
\begin{eqnarray}
\forall  (x,n) \in \mathbb{R} \times  \mathbb{N}, \quad   g(x\pm n) = \sum_{k=1}^3 | c_k |  e^{2 \pi i(\beta_kx+d_k)}g(x \pm n+\alpha_k).
\end{eqnarray}
Therefore, using the notation of Lemma 5.3, for each $(x,n) \in \mathbb{R} \times  \mathbb{N}$, we compute the following: 
\begin{eqnarray}
 \Delta_{03}(x,n) &=& | c_1 |  e^{2 \pi i(\beta_1x+d_1) }\Delta_{13}(x,n)\\&+& |c_2 | e^{2 \pi i(\beta_2x+d_2) }\Delta_{23}(x,n); \nonumber\\
 \Delta_{02}(x,n) &=&| c_1 |  e^{2 \pi i(\beta_1x+d_1) }\Delta_{12}(x,n)\\&-& |c_3 | e^{2 \pi i(\beta_3x+d_3) }\Delta_{23}(x,n); \nonumber 
\end{eqnarray}
and
\begin{eqnarray}
\Delta_{01}(x,n) &=& -| c_2 |  e^{2 \pi i(\beta_2x+d_2) }\Delta_{12}(x,n)\\ &-& |c_3 | e^{2 \pi i(\beta_3x+d_3) }\Delta_{13}(x,n).\nonumber
\end{eqnarray}

{\it viii.c.} Assume that $\beta_1 \neq 0$.

{\it viii.c.1.} If $\alpha_2 =\alpha_3$, then (5.18) is
\begin{eqnarray*}
 \Delta_{03}(x,n) = | c_1 |  e^{2 \pi i(\beta_1x+d_1) }\Delta_{13}(x,n).
\end{eqnarray*}
Therefore, we obtain the contradiction that $|\{x: e^{2 \pi i(\beta_1x+d_1) } \in \mathbb{R} \}| \neq 0$,
since, by (5.16), there is $n>0$ for which  we have $|\{x: \Delta_{03}(x,n) \}| \neq 0$,  and, by Proposition 1.1, $|\{x: e^{2 \pi i(\beta_1x+d_1) } \in \mathbb{R} \}| = 0$.

{\it viii.c.2.} If $\alpha_1 =\alpha_2<\alpha_3$, then (5.19) is
\begin{eqnarray*}
 \Delta_{02}(x,n) = -|c_3 | e^{2 \pi i(\beta_3x+d_3) }\Delta_{23}(x,n),
\end{eqnarray*}
and so, using a similar argument to the steps in case {\it viii.c.1.}, $\beta_3 \neq 0$ leads to a contradiction. Therefore, we can assert that  $\beta_3=0$, and so we also have $e^{2 \pi id_3}=-1$.

Meanwhile, (5.18) is
\begin{eqnarray*}
   \Delta_{03}(x,n) &=& [c_1  e^{2 \pi i\beta_1x }+ c_2  e^{2 \pi i\beta_2x }]\Delta_{23}(x,n),
\end{eqnarray*}
and so $|\{x: c_1e^{2 \pi i\beta_1x }+ c_2  e^{2 \pi i\beta_2x } >0 \}|\neq 0$, since, by (5.16), there is $n>0$ for which $|\{x: \Delta_{03}(x,n)>0 \}|\neq 0$. Therefore, using Proposition 1.1, we obtain that $c_2=\overline{c_1}$ and $\beta_2=-\beta_1$. Thus, in this case,
 for $x=(1/2 -d_1)/\beta_1$ and $n=0$, (5.17) leads to the contradiction,
\begin{eqnarray*}
   g(x) = -|c_1|  g(x+\alpha_1)-|c_1|g(x+\alpha_2) - |c_3|  g(x+\alpha_3).
\end{eqnarray*}

{\it viii.c.3.} If $0=\alpha_1 <\alpha_2<\alpha_3$, then (5.18) is
\begin{eqnarray*}
 [e^{-2 \pi i(\beta_2x+d_2)}-| c_1 |  e^{2 \pi i((\beta_1-\beta_2)x+d_1-d_2) }]\Delta_{03}(x,n)= |c_2 | \Delta_{23}(x,n),
\end{eqnarray*}
and so, using (5.16), we obtain  that
$$|\{x: e^{-2 \pi i(\beta_2x+d_2)}-| c_1 |  e^{2 \pi i((\beta_1-\beta_2)x+d_1-d_2) } \in \mathbb{R} \}| \neq 0.$$
Therefore, using Proposition 1.1 and the fact that $\beta_1 \neq 0$, we obtain that $\beta_1=2\beta_2$ and $c_1=-e^{4 \pi i d_2}$. Similarly, using (5.19), we obtain that $\beta_1=2\beta_3$. Therefore, we have $\beta_2=\beta_3$, and so, using(5.20) and (5.16), we obtain that
$e^{2 \pi id_3 }= - e^{2 \pi id_2 }$. Consequently, for $x=-d_2/\beta_2$,  (5.17) yields the equality
\begin{eqnarray}
2g(x-n) = |c_2|g(x -n+\alpha_2)- |c_3|g(x  -n+\alpha_3),  \nonumber
\end{eqnarray}
and so, for $n$ large enough, $|c_2|/|c_3|> g(x  -n+\alpha_3)/g(x -n+\alpha_2)\geq 1$. Meanwhile, for $x=(1/2-d_2)/\beta_2$,  (5.17) yields the equality
\begin{eqnarray}
2g(x+n) = -|c_2|g(x +n+\alpha_2)+ |c_3|g(x  +n+\alpha_3)],  \nonumber
\end{eqnarray}
and so, for $n$ large enough, $|c_2|/|c_3|< g(x +n+\alpha_3)/g(x +n+\alpha_2)\leq 1$. Thus, we obtain the contradiction, $1<|c_2|/|c_3|<1$.

{\it viii.c.4.} Assume that $0<\alpha_1 <\alpha_2<\alpha_3$. In this case, we assert that
\begin{eqnarray}
 \forall a \in E, \quad  \Delta_{23}(a,n)>0 \quad \mbox{ for each $n$ large enough,}
\end{eqnarray}
where $E=\{ x: e^{2 \pi i(\beta_1x+d_1) }$ is not a positive number$\}$. Indeed, if $a \in E$ and $\Delta_{23}(a,n)=0$ for some $n$ large enough, then, by (5.18), $\Delta_{03}(a,n)=0$, and so, by Lemma 5.3, $g$ is constant on $[a+n,a+n+\alpha_3]$. Therefore, by Lemma 5.5, the HRT conjecture holds for $\mathcal{G}(g, \{(-\alpha_k, \beta_k\}_{k=0}^3)$.

Now, for $x=(\pm 1/2 -d_1)/\beta_1)$, (5.18) is
\begin{eqnarray*}
   \Delta_{03}(x,n) =-| c_1 | \Delta_{13}(x,n) + |c_2 | e^{2 \pi i(\beta_2(\pm1/2 -d_1)/\beta_1)+d_2) }\Delta_{23}(x,n);
\end{eqnarray*}
and since, by (5.21), we have $\Delta_{23}(x,n)>0$  for  $n$ large enough,   then  we obtain that  $e^{2 \pi i d_2}=e^{2 \pi i\beta_2(d_1-1/2)/\beta_1}=e^{2 \pi i\beta_2(d_1+1/2)/\beta_1}$. Therefore, $e^{2 \pi i\beta_2/\beta_1}=1$, and so $\beta_2 =p\beta_1$, for some integer $p$.  Thus, for $x=(1/4-d_1)/\beta_1$, (5.18) is
\begin{eqnarray}
   \Delta_{03}(x,n) &=& | c_1 | \Delta_{13}(x,n)i   + | c_2 |\Delta_{23}(x,n)(-i)^p.
\end{eqnarray}
By (5.21), once again,  we have $ \Delta_{23}(x,n)>0$ for $n$ large enough, and so we also have $ \Delta_{03}(x,n), \Delta_{13}(x,n) >0$ for the same $n$. Therefore, in this case, (5.22) leads to a contradiction.

{\it viii.d.} Using similar arguments, we obtain a desired contradiction for each of the remaining cases of the set $\{\beta_1, \beta_2, \beta_3\}$.

\end{proof}


\section*{Acknowledgements}
The authors gratefully acknowledge the support from the AFOSR-MURI Grant FA9550-05-1-0443. The first named author is also appreciative of the support of ARO-MURI Grant W911NF-09-1-0383 and NGA Grant HM-1582-08-1-0009. The authors  also had the good fortune to obtain expert background and insights from Chris Heil. We note that the results herein were all obtained by early 2008. The authors were motivated to publish upon hearing a beautiful lecture in 2012 by Darrin Speegle reporting on his work ~\cite{Bow2} with Marcin Bownik on the HRT conjecture. Finally, we received invaluable assistance from Travis Andrews.


\end{document}